\documentclass[11pt,twoside]{amsart}

\usepackage[T1]{fontenc}
\usepackage[latin1]{inputenc}% Codage du fichier en ISO-Latin-1 (accents...)
\usepackage[english]{babel}%

\usepackage{amsmath,amsthm,amssymb,amsfonts,amscd, stmaryrd}
\usepackage{mathtools}
\usepackage{url}
\usepackage{color}
\usepackage{setspace}
\usepackage{enumerate}
\usepackage{pinlabel}
\usepackage{graphicx}
\usepackage{caption}
\usepackage{psfrag}
\usepackage{wrapfig}
\usepackage{mathrsfs}
\usepackage{MnSymbol}
\usepackage{tikz-cd}

\definecolor{grey}{rgb}{0.5,0.5,0.5}
\definecolor{rouge}{rgb}{1, 0, 0}
\definecolor{vert}{rgb}{0,0.75,0}
\definecolor{magenta}{rgb}{0.75, 0, 0.75}
\definecolor{bleu}{rgb}{0, 0.5, 1}

\theoremstyle{plain}
\newtheorem{theorem}{Theorem}
\newtheorem{proposition}[theorem]{Proposition}

\newtheorem{lemma}[theorem]{Lemma}

\newtheoremstyle{theoremwithref}{}{}{\itshape}{}{\bfseries}{.}{.5em}{#1 #2 #3}
\theoremstyle{theoremwithref}

\theoremstyle{definition}
\newtheorem{definition}[theorem]{Definition}
\newtheorem{example}[theorem]{Example}
\newtheorem{remark}[theorem]{Remark}

\numberwithin{theorem}{section}
\numberwithin{equation}{section}

\newcommand{\ie}{i.e.\ }
\newcommand{\eg}{e.g.\ }
\newcommand{\resp}{resp.\ }

\newcommand{\R}{\mathbb{R}}

\newcommand{\PP}{\mathbf{P}}

\newcommand{\RR}{\mathbf{R}}
\newcommand{\ZZ}{\mathbf{Z}}
\newcommand{\NN}{\mathbf{N}}
\newcommand{\QQ}{\mathbf{Q}}

\newcommand{\SL}{\mathrm{SL}}
\newcommand{\PGL}{\mathrm{PGL}}
\newcommand{\PSL}{\mathrm{PSL}}

\newcommand{\g}{\mathfrak{g}}

\newcommand{\aaa}{\mathfrak{a}}

\newcommand{\Hom}{\mathrm{Hom}}

\newcommand{\HH}{\mathbb{H}}

\newcommand{\arccosh}{\mathrm{arccosh}}
\newcommand{\arcsinh}{\mathrm{arcsinh}}

\newcommand{\slack}{\mathsf{X}}
\newcommand{\bslambda}{\boldsymbol{\lambda}}
\newcommand{\bsrho}{\boldsymbol{\rho}}
\newcommand{\bssigma}{\boldsymbol{\sigma}}

\title{Limit cones of multi-Fuchsian representations}

\author{Jeffrey Danciger}
\address{Department of Mathematics, The University of Texas at Austin, 1 University Station C1200, Austin, TX 78712, USA}
\email{jdanciger@math.utexas.edu}

\author{Fran\c{c}ois Gu\'eritaud}
\address{CNRS and IRMA, Universit\'e de Strasbourg, 7 rue Ren\'e Descartes, 67084 Strasbourg Cedex, France}
\email{francois.gueritaud@unistra.fr}

\author{Fanny Kassel}
\address{CNRS and Laboratoire Alexander Grothendieck, Institut des Hautes \'Etudes Scientifiques, Universit\'e Paris-Saclay, 35 route de Chartres, 91440 Bures-sur-Yvette, France}
\email{kassel@ihes.fr}

\thanks{J.D. was partially supported by the National Science Foundation under grant DMS 1945493.
Part of this work was completed while J.D. and F.G were in residence at the IHES in the spring 2023, while F.G. was in residence at UT Austin in the spring 2024, supported in part by the National Science Foundation under grant DMS 1937215, and while F.K. was in residence at the Institute for Advanced Study in Princeton in the spring 2024, supported by the National Science Foundation under grant DMS 1926686.}

\begin{document}

\maketitle

\begin{abstract}
We study the set of normalized multi-lengths for representations of closed surface groups and free groups into $(\mathrm{PSL}_2 \RR)^d$ whose projections to $\PSL_2 \RR$ are all convex cocompact.
These multi-lengths define a convex cone in $\RR^d_{\geq 0}$, called the limit cone.
When $d=3$, we show the coexistence of different regimes: for some representations the limit cone has only a finite number of sides, which we can force to grow like the genus (or free rank); for other representations, extremal rays are dense in the boundary of the limit cone.
We also give examples where the limit cone varies discontinuously with the representation.
\end{abstract}

\tableofcontents

%%%%%%%%%%%%%%%%%%%%%%%%%%%%%%%%%%%%%%%%%%%%%%%%%%%
\section{Introduction}

%%%%%%%%%%%%%%%%%%%%%%%%%
\subsection{The limit cone} \label{subsec:intro-lim-cone}

Let $\rho: \Gamma \to G$ be a representation of a discrete group $\Gamma$ into a real linear semisimple Lie group $G$. 
The \emph{limit cone} $\Lambda_\rho$ of $\rho(\Gamma)$ is a closed cone in a finite-dimen\-sional real vector space $\mathfrak{a} \simeq \RR^{\mathrm{rank}_{\RR}(G)}$, which records the range of possible spectral data achieved by elements of $\rho(\Gamma)$.
More precisely, $\Lambda_\rho$ is the closure of the cone spanned by $\bslambda(\rho(\Gamma))$, where $\bslambda: G \to \mathfrak{a}^+$ is the \emph{Jordan projection}, a conjugation-invariant map taking values in a closed positive Weyl chamber $\mathfrak{a}^+$ of~$\mathfrak{a}$.
An example to have in mind, and the main case we will explore in this paper, is that of the product $G = (\PSL_2 \RR)^d$ of $d \geq 2$ copies of $\PSL_2 \RR$.
In this case, the Jordan projection $\bslambda: (\PSL_2 \RR)^d \to \mathfrak{a}^+ \simeq \RR^d_{\geq 0} := [0,+\infty)^d$  is simply the product of $d$ copies of the translation length function $\lambda: \PSL_2 \RR \to \RR_{\geq 0}$ for the action of $\PSL_2 \RR$ on the hyperbolic plane $\HH^2$, and the limit cone $\Lambda_\rho$ of $\rho = (\rho_1, \ldots, \rho_d)$ is the closure of the cone spanned by the multi-lengths $\bslambda (\rho(\gamma)) = (\lambda(\rho_1(\gamma)), \ldots, \lambda(\rho_d(\gamma)))$, with $\gamma$ ranging over the group~$\Gamma$.

Benoist~\cite{ben97} showed that the limit cone $\Lambda_\rho$ is convex with nonempty interior as soon as $\rho$ is Zariski dense, which has some interesting applications (see \eg \cite{dk02,ben04,bcls15,zim21,sam24}); moreover, he showed that any closed convex cone of~$\mathfrak{a}^+$ with nonempty interior (and which is stable under the opposition involution) is achieved as the limit cone for some representation of an infinitely generated free group.
However, little is known about the possible shapes the limit cone may take for a fixed finitely generated group $\Gamma$ and semisimple Lie group~$G$.
The present paper is concerned with the basic question of how to compute $\Lambda_\rho$ and, to that end, how to certify that a particular ray lies in its boundary. 

Of particular interest is the case that $\Gamma = \pi_1(S)$ is the fundamental group of a finite-type surface $S$ of negative Euler characteristic, the group $G$ is real split, and the representation $\rho$ is positive in the sense of Fock--Goncharov~\cite{fg07}.
This class of representations is one of the main objects of study in the subject of higher Teichm\"uller theory (see \eg \cite{biw14,wie-icm}), and the shape of the limit cone for these representations would seem to be a powerful invariant, were there techniques for computing it.
However, it has proven extremely difficult to compute the limit cone, even in individual examples.
In the case that the rank of $G$ is at least three, it was not even known whether the limit cone is ever finite-sided, or ever infinite-sided.
Even in the case that the rank of $G$ is two, in which the limit cone is a cone on an interval, determining the precise endpoints of this interval seems to be difficult in general.

There is, however, one case for which techniques to compute the limit cone already exist, namely the case of $G = (\PSL_2 \RR)^2$.
The starting point for the present paper is the observation that the problem of computing limit cones for positive representations of (closed) surface groups into $(\PSL_2\RR)^2$ is exactly the problem of computing Thurston's asymmetric metric on Teichm\"uller space: the projectivized limit cone is precisely the closure of the set of length ratios considered by Thurston~\cite{thu-stretch-maps}.
The theory of the asymmetric metric makes computing such limit cones tractable in many cases.
Building on some of these ideas, the main purpose of this paper is to demonstrate a technique to compute the precise shape of the limit cone in some examples for the case $G = (\PSL_2 \RR)^d$ with $d \geq 3$.

%%%%%%%%%%%%%%%%%%%%%%%%%
\subsection{Main results} \label{subsec:intro-results}

We use the following terminology.

\begin{definition} \label{def:multi-Fuchsian}
Let $S$ be a finite-type surface of negative Euler characteristic and let $d\geq 1$ be an integer.
A representation $\bsrho = (\rho_1, \ldots, \rho_d): \pi_1(S) \to G = (\PSL_2 \RR)^d$ is \emph{multi-Fuchsian} if each component~$\rho_i$ is the holonomy representation of a complete hyperbolic structure on~$S$; we say that $\bsrho$ is \emph{multi-Fuchsian with no cusps} if this hyperbolic structure has no cusps for each~$i$.
\end{definition}

(If $\bsrho$ is multi-Fuchsian with no cusps, then the subgroups $\rho_i(\pi_1(S))$ of $\PSL_2 \RR$ are convex cocompact and $\bsrho$ is a positive representation in the sense of Fock--Goncharov \cite{fg07}.)

Let $\bslambda: G \to \mathfrak{a}^+ = \RR^d_{\geq 0}$ be the Jordan projection defined above.
Then the function $\bslambda_{\bsrho} = \bslambda \circ \bsrho : \pi_1(S) \to \RR^d_{\geq 0}$ assigns to each (free homotopy class of) closed curve $\gamma$ on~$S$, the vector $\bslambda_{\bsrho}(\gamma) = (\lambda_{\rho_1}(\gamma), \ldots, \lambda_{\rho_d}(\gamma))$ whose components are the geodesic lengths $\lambda_{\rho_i}(\gamma) := \lambda(\rho_i(\gamma))$ of $\gamma$ in each of the $d$ hyperbolic structures.
The limit cone $\Lambda_{\bsrho}$ is the closure of the cone spanned by the image of $\bslambda_{\bsrho}$.
When the component representations $\rho_i$ of a multi-Fuchsian representation $\bsrho$ are pairwise nonconjugate, $\Lambda_{\bsrho}$ is convex with nonempty interior (see Section~\ref{subsec:Z-dense-multi-Fuchsian}).
 
We present here a technique for explicitly determining the boundary of the limit cone for certain multi-Fuchsian representations. We obtain the following results in the case $d = 3$, demonstrating the coexistence of two very different regimes.

\begin{theorem} \label{thm:N-gon}
For any $g\geq 2$, let $S_g$ be a closed orientable surface of genus~$g$.
Then there exist multi-Fuchsian representations $\boldsymbol{\rho} : \pi_1(S_g) \to (\PSL_2\RR)^3$ whose projectivized limit cone $\PP \Lambda_{\boldsymbol{\rho}}$ is a finite polygon with at least $4g-1$ sides.
\end{theorem}

\begin{theorem} \label{thm:fishy}
Let $S$ be a one-holed torus.
Then there exist multi-Fuchsian representations $\boldsymbol{\rho} : \pi_1(S) \to (\PSL_2\RR)^3$ such that the projectivized limit cone $\PP \Lambda_{\boldsymbol{\rho}}$ is strictly convex, and every simple closed curve in~$S$ defines an extremal point in $\PP \Lambda_{\boldsymbol{\rho}}$.
These extremal points are all conical, and dense in the boundary of $\PP \Lambda_{\boldsymbol{\rho}}$.
\end{theorem}

\begin{remark}
In independent and simultaneous work, in the spirit of Theorem~\ref{thm:N-gon}, Alexander Nolte~\cite{nolte} also found a construction of finite-sided limit cones for multi-Fuchsian representations.
In particular, he showed that for $g\geq 2$, if $\Pi$ is any convex polyhedron of $\PP(\RR^d_{\geq 0})$ with at most $3g-3$ sides, and if $\Pi$ contains $[1:\dots:1]$ in its interior, then there is a multi-Fuchsian representation $\boldsymbol{\rho}:\pi_1(S_g) \to (\PSL_2\RR)^d$ with projectivized limit cone~$\Pi$.
\end{remark}

Theorem~\ref{thm:N-gon} can also be used to construct multi-Fuchsian representations $\boldsymbol{\rho} : \pi_1(S_g) \to (\PSL_2\RR)^3$ whose projectivized limit cone is a finite polygon with fewer sides: see Remark~\ref{rem:number-sides}.

\begin{remark}
The strictly convex limit cones of Theorem~\ref{thm:fishy} are reminiscent of the ``fish'' shapes identified by Bousch~\cite{bou00} and others in the context of ergodic optimization and the Hunt--Ott phenomenon~\cite{hunt-ott}: see \cite{boc-icm, jen19} for surveys.
\end{remark}

%%%%%%%%%%%%%%%%%%%%%%%%%
\subsection{Finding supporting hyperplanes to the limit cone} \label{subsec:intro-simple-lim-cone}

Our constructions for Theorems \ref{thm:N-gon} and~\ref{thm:fishy} both rely on a criterion (Theorem~\ref{thm:azimut-limit-cone}) that forces certain supporting hyperplanes to the limit cone (which we call \emph{azimuthal}) to be approached by the Jordan projections of \emph{simple} closed curves. 

More precisely, we introduce the following closed convex subcone of the limit cone.

\begin{definition} \label{def:simple-hull}
Let $S$ be a finite-type surface of negative Euler characteristic and $\bsrho : \pi_1(S) \to G = (\PSL_2 \RR)^d$ a representation.
The \emph{simple hull $\Lambda^s_{\bsrho}$ of~$\bsrho$ in $\mathfrak{a}^+$} is the closure in~$\mathfrak{a}^+$ of the convex hull of the cone spanned by $\bslambda_{\bsrho}(\pi_1(S)^s)$, where $\pi_1(S)^s$ is the set of elements of $\pi_1(S)$ corresponding to simple closed curves on~$S$.
\end{definition}

The simple hull $\Lambda^s_{\bsrho}$ is a closed convex subcone of~$\aaa^+$; it is contained in the limit set $\Lambda_{\bsrho}$ since $\Lambda_{\bsrho}$ is itself a closed convex subcone of~$\aaa^+$ (see Lemma~\ref{lem:conv-lim-cone}).

\begin{definition} \label{def:azim}
A closed half-space in $\mathfrak{a} \simeq \RR^d$ is called \emph{azimuthal} if it contains all but one corner rays of the closed Weyl chamber $\mathfrak{a}^+ \simeq \RR^d_{\geq 0}$, and its interior contains the central ray $\RR_{\geq 0} (1,\dots,1)$. We also call a co-oriented hyperplane in $\mathfrak a$, or a linear form on $\mathfrak a$, azimuthal if its associated closed positive half-space is azimuthal: see Figure~\ref{fig:azimuthal}.
\end{definition}
\begin{figure}[h!]
\captionsetup{width=0.9\linewidth}
\labellist
\small\hair 2pt
\pinlabel {$[1:0:0]$} [c] at        -0.5   0.5
\pinlabel {$[0:1:0]$} [c] at        3.9    5.75
\pinlabel {$[0:0:1]$} [c] at        6.8    0.5
\pinlabel {$[1:1:1]$} [c] at        3.15   2.6
\endlabellist
\includegraphics[width = 7cm]{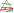}
\caption{Some projectivized co-oriented hyperplanes in $\PP(\RR^3_{\geq 0})$: we shade off their negative sides. Green: azimuthal. Red: non-azimuthal.}
\label{fig:azimuthal}
\end{figure}
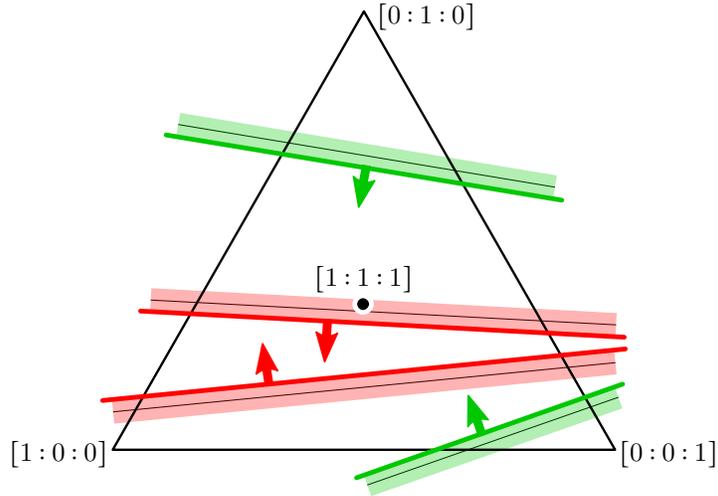

By convention, a supporting hyperplane to some convex cone $\mathcal C$ of $\mathfrak{a} \simeq \RR^d$ is co-oriented in the direction of~$\mathcal C$.
It therefore makes sense to ask whether a supporting hyperplane is azimuthal or whether an azimuthal co-oriented hyperplane is supporting.

We prove the following.

\begin{theorem} \label{thm:azimut-limit-cone}
For $d\geq 2$, let $S$ be a finite-type surface of negative Euler characteristic and $\bsrho = (\rho_1, \ldots, \rho_d): \pi_1(S) \to G = (\PSL_2 \RR)^d$ a multi-Fuchsian representation with no cusps, with Zariski-dense image in~$G$.
Then an azimuthal co-oriented hyperplane is supporting for~$\Lambda_{\bsrho}$ if and only if it is supporting for~$\Lambda_{\bsrho}^s$.
\end{theorem}

Theorem~\ref{thm:azimut-limit-cone} generalizes a theorem of Thurston \cite{thu-stretch-maps}, corresponding to the case $d = 2$ and $S$ closed.
In this case, as discussed above, the projectivized limit cone $\PP \Lambda_{\bsrho}$ identifies with the closure of the set $\{\lambda_{\rho_2}(\gamma)/\lambda_{\rho_1}(\gamma) ~|~\gamma \in \pi_1(S)\smallsetminus\{1\}\}$ of geodesic length ratios, which is an interval $[\alpha, \beta]$ containing $1$ in its interior.
Thurston showed that the extreme ratio $\beta$ is also the best Lipschitz constant for maps between the two hyperbolic structures in the correct homotopy class, and that the maximal stretching of an optimal Lipschitz map is realized along a \emph{geodesic lamination}. 
In particular, the endpoints of the interval $[\alpha, \beta]$ are approached by the length ratios of simple loops on~$S$.
By \cite[Th.\,1.3]{gk17}, this still holds when $S$ is not closed, as long as $(\rho_1,\rho_2)$ is multi-Fuchsian (possibly with cusps) and $1\in (\alpha, \beta)$; however, it can fail when $1\notin [\alpha,\beta]$ (so that one of the two supporting lines to $\Lambda_{\bsrho}$ in $\mathfrak{a} \simeq \RR^2$ is not azimuthal): see \cite[\S10.5]{gk17}.

We use Theorem~\ref{thm:azimut-limit-cone} to certify boundary points of the limit cone in the proofs of Theorems \ref{thm:N-gon} and~\ref{thm:fishy}. Indeed, it is a direct corollary that the limit cone $\Lambda_{\bsrho}$ and the simple hull $\Lambda_{\bsrho}^s$ coincide in the case that each point of the boundary of $\Lambda_{\bsrho}^s$ has an azimuthal supporting hyperplane. 
For certain surfaces of low complexity, for which the list of possible simple closed curves can be described exhaustively, we can directly arrange for $\Lambda_{\bsrho}^s$ to have azimuthal supporting hyperplanes and to be either finite-sided or strictly convex.

In general, computing the simple hull $\Lambda_{\bsrho}^s$ is a more tractable problem than computing the limit cone $\Lambda_{\bsrho}$ since the space of measured laminations is considerably smaller (finite dimensional) than the space of geodesic currents (infinite dimensional).
For a general multi-Fuchsian representation, $\Lambda_{\bsrho}$ will be larger than $\Lambda_{\bsrho}^s$, and without some tool like Theorem~\ref{thm:azimut-limit-cone} to reduce the complexity of the problem, it seems challenging to determine the shape of $\PP \Lambda_{\boldsymbol{\rho}}$ with precision or even to decide its finite-sidedness.

\begin{remark}
The finite-sided examples of Theorem~\ref{thm:N-gon} and~\cite{nolte} are produced by pinching some curves short. 
Although they are stable under small deformations, one should probably think of them as the ``exception''. 
For example, taking the $\rho_i$ infinitesimally close together and $d=1+\dim(\mathrm{Teich}(S))$, one may construct an infinitesimal limit cone $\PP \Lambda_\infty$ in $\mathrm{T}_{[1:...:1]} \PP(\RR^d_{\geq 0})$: this $\PP \Lambda_\infty$ is affinely equivalent to the closure $\mathcal{B}$ of the image of the map
$$\mathrm{d} \log \mathsf{length}: \pi_1(S) \longrightarrow \mathrm{T}^*_{[\rho_1]} \mathrm{Teich}(S)$$
which is at the heart of Thurston's work~\cite{thu-stretch-maps}.
Thurston showed that $\mathcal{B}$, also called the \emph{Lipschitz co-norm ball}, is convex, with images of simple closed curves dense in $\partial\mathcal{B}$, and an intricate (infinite) face structure (see \cite{pan23,hop25,bop}). 
Only certain low-dimensional projections of~$\mathcal{B}$, obtained by reducing $d$, will be finite-sided.
\end{remark}

In the same spirit, Theorems \ref{thm:N-gon} and~\ref{thm:fishy} should not be seen as a ``dichotomy'', as there may exist representations $\bsrho$ for which $\PP \Lambda_{\boldsymbol{\rho}}$ looks strictly convex in some regions but polyhedral in others.

%%%%%%%%%%%%%%%%%%%%%%%%%
\subsection{Discontinuous behavior of the limit cone}

Given a finitely generated group~$\Gamma$ and a semisimple Lie group~$G$, it is natural to wonder whether the limit cone $\Lambda_{\bsrho}$ varies continuously with $\bsrho\in\Hom(\Gamma,G)$.
It is straightforward to see that $\Lambda_{\bsrho}$ always varies lower semicontinuously: if a sequence $(\bsrho^{(n)})\in\Hom(\Gamma,G)^{\NN}$ converges to a representation~$\bsrho$, then any point of $\Lambda_{\bsrho}$ is a limit of points $x_n$ with $x_n\in\Lambda_{\bsrho^{(n)}}$ for all $n\in\NN$ (Remark~\ref{rem:lim-cone-lower-semicont}).
However, upper semicontinuity may fail, \ie $\Lambda_{\bsrho}$ may be smaller than any Hausdorff limit of the limit cones $\Lambda_{\bsrho^{(n)}}$.
An example of such behavior was given in \cite[\S\,10.6]{gk17}, for $G = (\PSL_2\RR)^2$: see Section~\ref{subsec:lim-cone-discont-d=2}.
Complementing this, in Section~\ref{subsec:lim-cone-discont-d=3} we use Theorem~\ref{thm:azimut-limit-cone} to give an example for $G = (\PSL_2\RR)^3$ where upper semicontinuity fails at a multi-Fuchsian representation $\bsrho = (\rho_1,\rho_2,\rho_3)$ with Zariski-dense image in~$G$, with $\rho_1$ and~$\rho_2$ convex cocompact but not~$\rho_3$ (Lemma~\ref{lem:lim-cone-discont-d=3}).
This gives an example of an Anosov representation of a Gromov hyperbolic group into a real semisimple Lie group, which has Zariski-dense image, and at which the limit cone does not vary continuously; it contrasts with a recent continuity result of Dey--Oh~\cite{do25} for Anosov representations, see Section~\ref{subsec:DeyOh}.

%%%%%%%%%%%%%%%%%%%%%%%%%
\subsection*{Acknowledgements}

We are grateful to the IH\'ES in Bures-sur-Yvette, to the IAS in~Prince\-ton, and to UT Austin for hosting us at various points during the completion of this project.
We thank Alex Nolte for discussion and for sharing on the advancement of his work.

%%%%%%%%%%%%%%%%%%%%%%%%%%%%%%%%%%%%%%%%%%%%%%%%%%%
\section{Preliminaries}

In the whole paper, $S$ denotes a finite-type orientable surface of negative Euler characteristic, with fundamental group $\pi_1(S)$.

%%%%%%%%%%%%%%%%%%%%%%%%%
\subsection{Reminders: geodesic currents}

Choose an arbitrary complete hyperbolic metric on~$S$ with no cusps.
This equips the unit tangent bundle $\mathrm{T}^1 S$ with a geodesic flow $\phi=(\phi_t)_{t\in \RR}$ and a reversal map $\epsilon:(x,v)\mapsto (x,-v)$.
By definition, a \emph{geodesic current} on~$S$ is a finite, $\phi$-invariant, $\epsilon$-invariant, Borel measure $\mu$ on $\mathrm{T}^1 S$.
We say that $\mu$ \emph{self-intersects} if its support in $\mathrm{T}^1 S$, seen modulo $\epsilon$, projects non-injectively to~$S$.
The space $\mathsf{Curr}(S)$ of all geodesic currents has the structure of a convex cone.
Positively-weighted closed geodesics form a subset of $\mathsf{Curr}(S)$ that is dense for the weak-$*$ topology.
The projectivization $\PP \mathsf{Curr}(S)$ is compact.
Examples of geodesic currents include the \emph{measured laminations} $\mathscr{ML}(S)$, which are those currents that do not self-intersect, \ie the limits of weighted \emph{simple} multicurves. 
The space $\mathscr{ML}(S)$ is much smaller than $\mathsf{Curr}(S)$; it is a finite-dimensional submanifold.

Topologically, the space $\mathsf{Curr}(S)$ does not depend on our choice of complete hyperbolic metric on~$S$ with no cusps: if $\texttt{g}_1$ and~$\texttt{g}_2$ are two such metrics, then there is a natural $1$-homogeneous homeomorphism between the set of geodesic currents defined with respect to~$\texttt{g}_1$ and the set of geodesic currents defined with respect to~$\texttt{g}_2$, taking the Lebesgue measure on a closed geodesic of $(S,\texttt{g}_1)$ to the Lebesgue measure on the unique closed geodesic of $(S,\texttt{g}_2)$ in the same homotopy class. 

However, determined by the choice of a complete hyperbolic metric on~$S$ with no cusps, each geodesic current $\mu$ has a well-defined \emph{length} $\int_{\mathrm{T}^1 S} 1\, \mathrm{d}\mu$; the length functional is $1$-homogeneous and continuously extends the lengths of closed geodesics with respect to the chosen hyperbolic metric.
In this way, given a multi-Fuchsian representation $\boldsymbol{\rho} = (\rho_1, \ldots, \rho_d): \pi_1(S) \to (\PSL_2 \RR)^d$ with no cusps and a current $\mu \in \mathsf{Curr}(S)$, we may talk about the Jordan projection $\boldsymbol{\lambda_\rho}(\mu) \in \mathfrak{a}^+ =\RR^d_{\geq 0}$.
This map $\boldsymbol{\lambda_\rho} : \mathsf{Curr}(S) \to \mathfrak{a}^+$ is continuous and $1$-homogeneous, and extends the map on closed geodesics induced by the conjugation-invariant map $\boldsymbol{\lambda_\rho} = \bslambda \circ \bsrho : \pi_1(S) \to \RR^d_{\geq 0}$ of Section~\ref{subsec:intro-results}. It also depends continuously on $\bsrho$.

Alternatively, following \cite{bon86}, it is possible to define geodesic currents without reference to a flow, directly as $\pi_1(S)$-invariant measures on
$$\mathscr G(\widetilde S) := \{(x,y)\in (\partial_\infty \pi_1(S))^2 ~|~ x\neq y\}/_{(x,y)\sim (y,x)}$$
(the space of unoriented lines of~$\widetilde{S}$) which are finite modulo $\pi_1(S)$.
In this formulation, the length of a current $\mu$ in a hyperbolic metric $\texttt{g}$ is obtained by integrating the natural product measure $d \mu \times  dt$ on $T^1S$, where $d t$ is the length along the flow direction determined by the metric $\texttt{g}$.
It is also natural to define the \emph{intersection number} $i(\mu, \mu')$ of two geodesic currents $\mu, \mu'$ as the integral of $d \mu\,d \mu'$ over the subset of $\mathscr G(\widetilde S) \times \mathscr G(\widetilde S)$ consisting of pairs of crossing lines in $\widetilde S$.

%%%%%%%%%%%%%%%%%%%%%%%%%
\subsection{Basic facts about multi-Fuchsian representations} \label{subsec:Z-dense-multi-Fuchsian}

The following statement is well known to experts, but we could not find it in the literature for general $d\geq 2$, so we include it here for the reader's convenience.
See \cite{dk00} for $d=2$.

\begin{lemma} \label{lem:Goursat}
Let $\bsrho = (\rho_1, \ldots, \rho_d): \pi_1(S) \to G = (\PSL_2 \RR)^d$ be a multi-Fuchsian representation.
Then the image of $\bsrho$ is Zariski-dense in~$G$ if and only if $\rho_1,\dots,\rho_d$ are pairwise distinct modulo the action of $\PGL_2 \RR$ on $\Hom(\pi_1(S),\PSL_2 \RR)$ by conjugation at the target.
\end{lemma}

In this case, the limit cone $\Lambda_{\bsrho}$ is convex with nonempty interior, by \cite{ben97}.

\begin{proof}
If some of the $\rho_i$ are conjugate by elements of $\PGL_2 \RR$, then the image of $\bsrho$ cannot be Zariski-dense in~$G$.

Conversely, let us prove by induction on $d\geq 1$ that if $\rho_1,\dots,\rho_d \in \Hom(\pi_1(S),\PSL_2 \RR)$ are injective with discrete image, and pairwise nonconjugate by elements of $\PGL_2 \RR$, then the image of $\bsrho = (\rho_1,\dots,\rho_d) : \pi_1(S)\to G = (\PSL_2 \RR)^d$ is Zariski-dense in~$G$.

The case $d=1$ is classical.
Recall that any proper algebraic subgroup of $\PSL_2 \RR$ is contained in a conjugate of $\mathrm{PSO}(2)$ or a conjugate of the group of upper triangular matrices.
Since $\rho_1$ is injective with discrete image and $\pi_1(S)$ is not finite nor virtually cyclic (because $S$ has negative Euler characteristic), the group $\rho_1(\pi_1(S))$ cannot be contained in such a subgroup of $\PSL_2 \RR$.
Therefore $\rho_1(\pi_1(S))$ is Zariski dense in $\PSL_2 \RR$.

Suppose the property is true for~$d$, and let us prove it for $d+1$.
Let $\rho_1,\dots,\rho_{d+1} \in \Hom(\pi_1(S),\PSL_2 \RR)$ be injective with discrete image, and pairwise nonconjugate by elements of $\PGL_2 \RR$.
Let $H$ be the Zariski closure of $(\rho_1,\dots,\rho_{d+1})(\pi_1(S))$ in $G = (\PSL_2 \RR)^{d+1}$.
We write $G = G'\times G''$ where $G'$ is the product of the first $d$ factors of $(\PSL_2 \RR)^{d+1}$, and $G''$ the last factor.
Let $\mathrm{pr}_{G'} : G\to G'$ and $\mathrm{pr}_{G''} : G\to G''$ be the corresponding projections.
By induction, the image of $(\rho_1,\dots,\rho_d)$ is Zariski-dense in $(\PSL_2 \RR)^d$, hence $\mathrm{pr}_{G'}(H) = G'$, and the image of $\rho_{d+1}$ is Zariski-dense in $\PSL_2 \RR$, hence $\mathrm{pr}_{G''}(H) = G''$.
The kernel $N''$ of the restriction $(\mathrm{pr}_{G'})|_H$ is a normal algebraic subgroup of $G'' \simeq \PSL_2 \RR$, hence equal to $G''$ or trivial.
The kernel $N'$ of the restriction $(\mathrm{pr}_{G''})|_H$ is a normal algebraic subgroup of $G' = (\PSL_2 \RR)^d$, hence equal to a product of $\PSL_2 \RR$ factors, because each is simple with trivial center.
We can write $G' = N' \times P$ where $P$ is the product of the other $\PSL_2 \RR$ factors.
If $N'' = G''$, then $H = G'\times G'' = G$.
Suppose by contradiction that $N''$ is trivial.
By the Goursat lemma (see \eg \cite{pet09}), we have $H = N' \times \mathrm{graph}(f)$ where $f : P\to G''$ is an isomorphism.
In particular, $P$ is one factor of $G' = (\PSL_2 \RR)^d$, say the $i$-th one, and the corresponding representations $\rho_i$ and $\rho_{d+1}$ are conjugate: contradiction.
\end{proof}

\begin{lemma} \label{lem:conv-lim-cone}
Let $\bsrho = (\rho_1, \ldots, \rho_d): \pi_1(S) \to G = (\PSL_2 \RR)^d$ be any multi-Fuchsian representation.
Then the limit cone $\Lambda_{\bsrho}$ is convex.
\end{lemma}

\begin{proof}
Choose integers $1\leq j_1 < \dots < j_n\leq d$ such that each $\rho_i$, for $1\leq i\leq d$, is conjugate by some element of $\PGL_2 \RR$ to $\rho_{j_{\ell_i}}$ for a unique integer $\ell_i \in \{1,\dots,n\}$, and $\ell_1,\dots,\ell_d$ take all possible values between $1$ and~$n$.
By Lemma~\ref{lem:Goursat}, the representation $\bsrho' := (\rho_{j_1},\dots,\rho_{j_n})$ has Zariski-dense image in $(\PSL_2 \RR)^n$, hence its limit cone $\Lambda_{\bsrho'}$ is convex with nonempty interior in $\RR^n_{\geq 0}$ by \cite{ben97}.
To conclude, we observe that the limit cone $\Lambda_{\bsrho}$ is the image of $\Lambda_{\bsrho'}$ by the injective linear map from $\RR^n$ to $\RR^d$ sending $(x_1,\dots,x_n)$ to $(x_{\ell_1},\dots,x_{\ell_d})$.
\end{proof}

%%%%%%%%%%%%%%%%%%%%%%%%%
\subsection{Asymptotic slack} \label{subsec:as-slack}

Recall that an embedding $\xi : \partial_{\infty} \pi_1(S) \to \PP^1\RR$ is called \emph{positive} if it sends any cyclically ordered tuple to a cyclically ordered tuple.
For instance, if $\rho : \pi_1(S)\to\PSL_2\RR$ is the holonomy of a complete hyperbolic metric on~$S$ with no cusps, then there is a (unique) $\rho$-equivariant continuous map $\xi : \partial_{\infty} \pi_1(S) \to \PP^1\RR$ sending the attracting fixed point of any $\gamma\in\Gamma\smallsetminus\{1\}$ to the attracting fixed point of $\rho(\gamma)$ in $\PP^1\RR$; this map, called the \emph{boundary map} of~$\rho$, is a positive embedding.

For any positive embedding $\xi : \partial_{\infty} \pi_1(S) \to \PP^1\RR$ and any four pairwise distinct points $x^-,y^-,x^+,y^+ \in \partial_{\infty} \pi_1(S)$, we set
\begin{equation} \label{eqn:slack-Gromov-bound}
\slack_{\xi}(x^-,y^-,x^+,y^+) := \log\, \big |[\xi(x^-) : \xi(y^-) : \xi(x^+) : \xi(y^+)] \big | \in \RR
\end{equation}
where $[\, \underline{~} : \underline{~} : \underline{~} : \underline{~}\,]$ is the usual cross ratio, normalized so that $[\infty:0:1:t] = t$.
In terms of the ordered pairs $\ell:=(x^-, x^+)$ and $\ell':=(y^-,y^+)$, also called (oriented) \emph{flow lines}, we will refer to
\begin{equation} \label{eqn:slack-Gromov-bound2}
\slack_\xi(\ell, \ell') := \slack_{\xi}(x^- ,y^-, x^+, y^+) \, \in \RR  
\end{equation}
as the \emph{asymptotic slack} of $\ell, \ell'$ with respect to~$\xi$.
We have $\slack_\xi(\ell, \ell')=\slack_\xi(\ell', \ell)=\slack_\xi(-\ell,-\ell')$ where the minus signs mean reversed orientations.
If $x^-, y^-, x^+, y^+$ are cyclically ordered in $\partial_{\infty} \pi_1(S)$, then the oriented geodesics $\xi(\ell):=(\xi(x^-), \xi(x^+))$ and $\xi(\ell'):=(\xi(y^-), \xi(y^+))$ in $\HH^2$ intersect at an angle $\theta \in (0,\pi)$ and
\begin{equation} \label{eqn:slack-angle}
\left \{ \begin{array}{lll}
\slack_{\xi}(\ell, \phantom{-}\ell') & \!\!\!=\!\!\! & - \log \cos^2(\theta/2) \ \geq 0, \\
\slack_{\xi}(\ell, -\ell') & \!\!\!=\!\!\! & - \log \sin^2(\theta/2) \ \,  \geq 0. \end{array} \right .
\end{equation}
We then say that $\ell, \ell'$ are \emph{crossing} flow lines.
The term ``slack'' is motivated by the fact that for any sequences $\mathsf{x}_n^\pm \to \xi(x^\pm)$ and $\mathsf{y}_n^\pm \to \xi(y^\pm)$ in $\HH^2$,  
\begin{equation} \label{eqn:slack-rope}
\slack_\xi(\ell, \ell')=\tfrac12  \lim_{n\to +\infty} 
\big( d(\mathsf{x}_n^-, \mathsf{x}_n^+)+d(\mathsf{y}_n^-, \mathsf{y}_n^+) - d(\mathsf{x}_n^-, \mathsf{y}_n^+) - d(\mathsf{y}_n^-, \mathsf{x}_n^+) \big),
\end{equation}
so $\slack_\xi(\ell, \ell')$ is the amount of ``slack'' saved by uncrossing the ``ropes'' $\xi(\ell), \xi(\ell')$.

When $\ell$ and $\ell'$ do not cross, we can also express $\slack_\xi(\ell, \ell')$ in terms of the length $L$ of the common perpendicular segment to $\xi(\ell)$ and $\xi(\ell')$ in $\HH^2$: if $\ell, \ell'$ are oriented consistently, \ie $x^-, y^-, y^+, x^+$ are cyclically ordered, then
\begin{equation*} %\label{eqn:slack-Hangle}
\left \{ \begin{array}{lll}
\slack_{\xi}(\ell, \phantom{-}\ell') & \!\!\!=\!\!\! & - \log \cosh^2(L/2) \quad \leq 0, \\
\slack_{\xi}(\ell, -\ell') & \!\!\!=\!\!\! & - \log \sinh^2(L/2). \end{array} \right .
\end{equation*}

%%%%%%%%%%%%%%%%%%%%%%%%%%%%%%%%%%%%%%%%%%%%%%%%%%%
\section{A simpleness criterion} \label{sec:simplecur}

In this section we prove Theorem~\ref{thm:azimut-limit-cone}.
The ``only if'' direction is paraphrased as follows (see Definitions \ref{def:multi-Fuchsian}, \ref{def:simple-hull}, and~\ref{def:azim}).

\begin{theorem} \label{thm:azimut-limit-cone-switched}
Let $S$ be a complete orientable hyperbolic surface of finite type.
For $d\geq 2$, let $\boldsymbol{\rho} = (\rho_1, \ldots, \rho_d): \pi_1(S) \to G = (\PSL_2 \RR)^d$ be a multi-Fuchsian representation with no cusps, with Zariski-dense image in~$G$.
Then any azimuthal supporting hyperplane to the limit cone $\Lambda_{\bsrho}$ meets the closure of $\bslambda_{\bsrho}(\pi_1(S)^s)$, hence the simple hull~$\Lambda_{\bsrho}^s$.
\end{theorem}

We start with three lemmas.

%%%%%%%%%%%%%%%%%%%%%%%%%
\subsection{Azimuthality and slacks}

Recall the notation \eqref{eqn:slack-Gromov-bound}--\eqref{eqn:slack-Gromov-bound2}, and the notions of positive map and crossing flow lines from Section~\ref{subsec:as-slack}.

\begin{lemma} \label{lem:Omega-non-pos}
Let $\Omega = \sum_{i=1}^d c_i e_i^* \in \mathfrak{a}^*$ be azimuthal, with the closed halfspace $\{\Omega \geq 0\}$ containing all  corner rays of the Weyl chamber $\mathfrak{a}^+\simeq \RR^d_{\geq 0}$ but the $i_0$-th one. 
Then there exists $c>0$ such that for any positive maps $\xi_1,\dots,\xi_d : \partial_{\infty} \pi_1(S) \to \PP^1\RR$ and any crossing flow lines $\ell, \ell' \in (\partial_{\infty} \pi_1(S))^2$, the $\RR^d$-vectors $\boldsymbol{X} := (\slack_{\xi_i}(\ell, \ell'))_{1\leq i\leq d}$ and $\boldsymbol{X}' := (\slack_{\xi_i}(\ell, -\ell'))_{1\leq i\leq d}$ satisfy
\begin{equation} \label{eq:conveX}
\Omega(\boldsymbol{X}) \geq c \,\slack_{\xi_{i_0}}(\ell, \ell')\quad \text{or} \quad \Omega(\boldsymbol{X}') \geq c \,\slack_{\xi_{i_0}}(\ell, -\ell').
\end{equation}
\end{lemma}

\begin{proof}
Up to permuting coordinates we may assume $i_0=1$, in other words $c_1<0\leq c_2, \dots, c_d$.
Since $\ell$ and~$\ell'$ are crossing flow lines, for any $1\leq i\leq d$ we can write $\slack_{\xi_i}(\ell, \ell') = -\log\sigma_i$ for some $\sigma_i \in (0,1)$; then $\slack_{\xi_i}(\ell, -\ell') = -\log(1-\sigma_i)$ by \eqref{eqn:slack-angle}.

Let $C:=\sum_{i\geq 2} c_i$ and $\overline{\sigma} := (\sum_{i\geq 2} c_i \sigma_i)/C$. 
We distinguish two cases: if $\overline{\sigma} \leq \sigma_1$, then
\begin{align*} \Omega(\boldsymbol{X}) & =  \textstyle -c_1  \log \sigma_{1} - \sum_{i \geq 2} c_i \log \sigma_i \\ 
& \geq  -c_1 \log \sigma_{1} - C \log\overline{\sigma} & \text{by concavity of $\log$} \\
& \geq - (c_1+C) \log \sigma_1 & \text{by monotonicity of $\log$} \\
& = (c_1+C)\, \slack_{\xi_{1}}(\ell, \ell').
\end{align*}
Similarly, if $\overline{\sigma} \geq \sigma_1$, then
\begin{align*} \Omega(\boldsymbol{X}') & \textstyle = -c_1 \log(1-\sigma_{1}) - \sum_{i \geq 2} c_i \log(1-\sigma_i) \\ 
& \geq -c_1 \log(1-\sigma_{1}) - C \log (1-\overline{\sigma}) & \phantom{AAAAAAAA}\\
& \geq - (c_1 + C) \log (1-\sigma_1) \\
& = (c_1 + C)\, \slack_{\xi_{1}}(\ell, -\ell').
\end{align*}
Hence \eqref{eq:conveX} holds with $c:=c_1+C=\Omega(1,\dots, 1)$, which is positive because $\Omega$ is azimuthal.
\end{proof}

%%%%%%%%%%%%%%%%%%%%%%%%%
\subsection{Surgeries on long curves and slacks} \label{subsec:prelim-hyperbolic}

Let $\rho : \pi_1(S) \to \PSL_2\RR$ be the holonomy representation of a complete hyperbolic structure on~$S$ with no cusps, and let $\xi : \partial_{\infty} \pi_1(S) \to \PP^1 \RR$ be the corresponding equivariant boundary map.
Consider an element $\gamma \in \pi_1(S)$ that has two conjugates whose images under~$\rho$ have oriented translation axes $\ell, \ell'$ in $\HH^2$ intersecting at an angle $\theta \in (0,\pi)$.
The intersection point $p$ corresponds to a point of self-crossing of the geodesic representative of the closed curve $[\gamma]$ on the surface $S$.
Resolving this crossing in one way yields one curve $[\gamma']$ and in the other way yields two curves $[\gamma''], [\gamma''']$ (see Figure~\ref{fig:resolution}).
The length $L = \lambda_{\rho}(\gamma) > 0$ of the geodesic representative of $[\gamma]$ splits as $L=L''+L'''$ before and after the self-crossing.
The following elementary estimate uses notation \eqref{eqn:slack-Gromov-bound}--\eqref{eqn:slack-Gromov-bound2}.

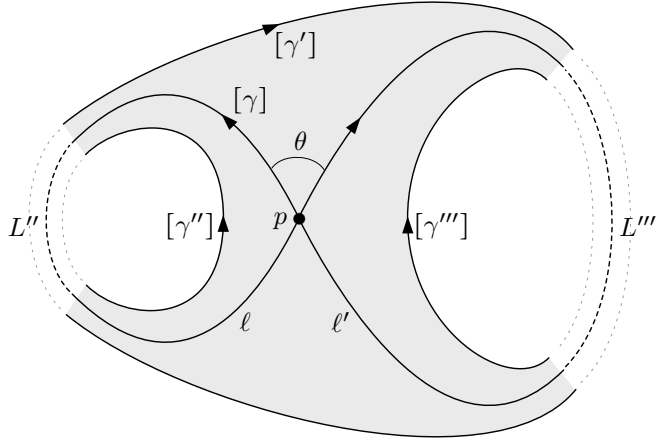
\begin{figure}[h!] 
\captionsetup{width=0.9\linewidth}
\labellist
\small\hair 2pt
% 		G R I D    F O R    P L A C I N G    L A B E L S
%\grille{0} \grille{2} \grille{4} \grille{6} \grille{8} \grille{10} \grille{12} \grille{14} \grille{16} \grille{18}
%\grille{20} \grille{22} \grille{24} \grille{26} \grille{28} \grille{30} \grille{32} \grille{34} \grille{36} \grille{38} 
%\grille{40} \grille{42} \grille{44} \grille{46} %\grille{48} \grille{50} \grille{52} \grille{54} \grille{56} \grille{58} 
%
\pinlabel {$\theta$} [c] at		5.4	5.9
\pinlabel {$[\gamma]$} [c] at	4.4	6.7
\pinlabel {$[\gamma']$} [c] at	5.3	7.8
\pinlabel {$[\gamma'']$} [c] at	3.2	4.2
\pinlabel {$[\gamma''']$} [c] at	8.2	4.2
\pinlabel {$L''$} [c] at			-0.1	4.2
\pinlabel {$L'''$} [c] at		12.1	4.2
\pinlabel {$\ell$} [c] at		4.3	2.3
\pinlabel {$\ell'$} [c] at		6.2	2.3
\pinlabel {$p$} [c] at			5.0	4.3
%%%%%%%%%%%%%%%%%%%%%%%%%
\endlabellist
\includegraphics[width = 8cm]{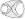}
\caption{Resolutions $[\gamma']$ and $[\gamma''] \cup [\gamma''']$ of a self-crossing of $[\gamma]$ in an immersed pair of pants in $\rho(\pi_1(S)) \backslash \mathbb{H}^2 \simeq S$}
\label{fig:resolution}
\end{figure}

\begin{lemma} \label{lem:thinpants}
With the above notation, for any $\varepsilon, \theta_0>0$, there exists $L_0>0$ such that if $\theta \in [\theta_0, \pi-\theta_0]$ and $L'', L''' \geq L_0$, then
\begin{align*}
\left |\lambda_{\rho}(\gamma'') + \lambda_{\rho}(\gamma''') - \left ( \lambda_{\rho}(\gamma) - 2\,\slack_{\xi}(\ell, \phantom{-}\ell') \right ) \right | & \leq \varepsilon; \\
\left | \lambda_{\rho}(\gamma') - \left (\lambda_{\rho}(\gamma) - 2\,\slack_{\xi}(\ell, -\ell') \right ) \right | & \leq \varepsilon.
\end{align*}
\end{lemma}

\begin{proof}
From hyperbolic trigonometry in a pair of pants, we have
\begin{align*}
\cosh \tfrac{\lambda_{\rho}(\gamma')}{2} & = 2 \cosh \tfrac{L''}{2} \cosh \tfrac{L'''}{2} \sin^2 \tfrac{\theta}{2}  - \cosh \tfrac{L''-L'''}{2} \,, \\ 
\cosh \tfrac{\lambda_{\rho}(\gamma'')}{2} & = \cosh \tfrac{L''}{2} \cos \tfrac{\theta}{2} \,, \\ 
\cosh \tfrac{\lambda_{\rho}(\gamma''')}{2} & = \cosh \tfrac{L'''}{2} \cos \tfrac{\theta}{2} \,.
\end{align*}
The result follows by taking $L'', L''' \gg 1$, using \eqref{eqn:slack-angle} and approaching $\cosh$, $\sinh$ by $\frac{1}{2} \exp$.
\end{proof}

%%%%%%%%%%%%%%%%%%%%%%%%%%
\subsection{Self-intersecting currents and long crossing curves}

\begin{lemma} \label{lem:pigeon-hole}
Let $\rho:\pi_1(S) \to \PSL_2\RR$ be the holonomy representation of a complete hyperbolic structure on~$S$ with no cusps. 
Let $([\gamma_n])_{n\in \NN}$ be a sequence of closed curves in $S$, converging in $\mathsf{Curr}(S)$ to a self-intersecting geodesic current $\mu$, and satisfying $L(n):=\lambda_\rho(\gamma_n) \to +\infty$. 
Let $\ell_n \subset \HH^2$ be the oriented translation axis of $\rho(\gamma_n)$.
Then there exist elements $\beta_n\in \pi_1(S)$ and a constant $\theta_0>0$ such that for large $n$, the oriented geodesic lines $\ell_n$ and $\ell_n':=\rho(\beta_n)\cdot \ell_n$ cross at a point $p_n\in \HH^2$ such that: 
\begin{itemize}
  \item $\measuredangle_{p_n}(\ell_n, \ell'_n)  \in [\theta_0, \pi-\theta_0]$;
  \item if $I_n \subset \ell_n$ denotes the segment $[p_n, \rho(\gamma_n)\cdot p_n]$, of length $L(n)$, then $p_n$ lies on the segment $\rho(\beta_n)\cdot I_n \subset \ell'_n$, at distances $L''(n) \to +\infty$ and $L'''(n)\to +\infty$ from its endpoints.
\end{itemize}
\end{lemma}

The conclusion of Lemma~\ref{lem:pigeon-hole} can be rephrased simply as saying that we can choose self-crossings that satisfy the assumptions of Lemma~\ref{lem:thinpants}.

\begin{proof}[Proof of Lemma~\ref{lem:pigeon-hole}]
Since the geodesic current $\mu\in \mathsf{Curr}(S)$ self-intersects, its support lifts to a subset of $\mathrm{T}^1\mathbb H^2$ containing two flow lines that project to a pair of crossing geodesics $\ell_\infty, \ell'_\infty \subset \mathbb{H}^2$. 
In particular, for $n$ large enough, $\gamma_n$ has two conjugates whose axes cross near $\ell_\infty \cap \ell'_\infty$.
We split the length $L(n)=\lambda_{\rho}(\gamma_n) > 0$ as $L(n)=L''(n)+L'''(n)$ as in Section~\ref{subsec:prelim-hyperbolic} above (see Figure~\ref{fig:resolution}).

Let $v,v' \in \mathrm{T}^1\HH^2$ be the respective tangent vectors to $\ell_\infty, \ell'_\infty$ at the crossing point.
Since $\ell_n$ and~$\ell'_n$ converge respectively to $\ell$ and~$\ell'$, for any $\varepsilon > 0$ there exists $\delta > 0$ such that for any sufficiently large~$n$, the trajectory of $\gamma_n$ in $\mathrm{T}^1S$ passes $\varepsilon$-close to each of $[v]$ and $[v']$
at least $\delta L(n) = \delta \lambda_{\rho}(\gamma_n)$ times. 
We fix $\varepsilon > 0$ small enough that there exists $\theta_0>0$  such that the lifts of any two of the close passes, near $v$ and $v'$, intersect in $\mathbb{H}^2$ at an angle in $[\theta_0, \pi-\theta_0]$.

Further, there is some definite amount of length $r > 0$ between each successive pass near $[v]$ and likewise for~$[v']$. 
The trajectory of $\gamma_n$ in $\mathrm{T}^1 S$ therefore makes an $\varepsilon$-close pass to~$[v]$, defining the lift $\ell_n$, and an $\varepsilon$-close pass to $[v']$,  defining the lift $\ell_n'=\beta_n\cdot \ell_n$, so that the lengths $L''(n)$, $L'''(n)$ traveled between these two passes in either direction are both at least $r \delta L(n)/3$,  which goes to infinity as $n \to+ \infty$. 
\end{proof}

%%%%%%%%%%%%%%%%%%%%%%%%%
\subsection{Proof of Theorem~\ref{thm:azimut-limit-cone-switched}}

Let $\mathcal{H}$ be a supporting hyperplane to~$\Lambda_{\bsrho}$ which is azimuthal: we have $\mathcal{H} = \mathrm{Ker}(\Omega)$ and $\Omega(\Lambda_{\bsrho})\geq 0$ where $\Omega = \sum_{i=1}^d c_i e_i^*$ for some $c_i\in\RR$ which are all $\geq 0$ except for one, which we can take to be $c_1<0$ up to renumbering.

Consider a sequence $(\gamma_n) \in (\pi_1(S)\smallsetminus \{1\})^{\NN}$ such that 
$\lambda_{\rho_1}(\gamma_n) \to +\infty$ (hence $\lambda_{\rho_i}(\gamma_n) \to +\infty$ for all $1\leq i\leq d$) and such that the direction of $\boldsymbol{\lambda_\rho}(\gamma_n)$ converges into $\mathrm{Ker}(\Omega)$. 
Further, choose these $\gamma_n$ to be \emph{record breakers} with respect to $\lambda_{\rho_1}$: by this we mean that for any~$n\in\NN$,
\begin{equation} \label{eq:record}
m_n := \frac{\Omega(\boldsymbol{\lambda_\rho}(\gamma_n))}{\lambda_{\rho_1}(\gamma_n)} \leq \frac{\Omega(\boldsymbol{\lambda_\rho}(\beta))}{\lambda_{\rho_1}(\beta)} ~\text{ for all }~ \beta \in \pi_1(S)\smallsetminus \{1\} ~\text{ with }~ \lambda_{\rho_1}(\beta) \leq \lambda_{\rho_1}(\gamma_n).
\end{equation}
Such an infinite sequence always exists (possibly ending with powers $\gamma_n=\gamma^{k_n}$ of a single element~$\gamma$).
After taking a subsequence, the closed curves $[\gamma_n]$ converge projectively to a geodesic current $\mu \in \mathsf{Curr}(S)$ such that $\Omega(\bslambda_{\bsrho}(\mu)) = 0$. 

We claim that $\mu$ does not self-intersect, \ie it is a measured lamination. 
Indeed, suppose by contradiction that $\mu$ does self-intersect.
Let $\theta_0>0$ be given by Lemma~\ref{lem:pigeon-hole} applied to $([\gamma_n])_{n\in \NN}$ and (say) $\rho_1$; and define $L_1(n):=\min\{L''(n), L'''(n)\} \to +\infty$ in the notations of that lemma. 
Let $L_0(n)$ be given by Lemma~\ref{lem:thinpants} for $\rho_1$ and the pair $(\theta_0, \varepsilon(n))$, where $\varepsilon(n)\to 0$ slowly enough that $L_0(n)\leq L_1(n)$.
To summarize, Lemma~\ref{lem:pigeon-hole} provides conjugates of $\gamma_n$ with crossing axes $\ell_n, \ell'_n$ satisfying all assumptions of Lemma~\ref{lem:thinpants}, including $L''(n),L'''(n) \geq L_0(n)$; and, up to taking a subsequence,  $\ell_n, \ell'_n$ converge to crossing axes $\ell, \ell'$. 
This is true (up to adjusting~$\theta_0$ and $\varepsilon(n)\to 0$) for all representations $\rho_i$ simultaneously, since the $d$ hyperbolic surfaces $\rho_i(\pi_1 S) \backslash \HH^2$ are mutually quasi-isometric.
Let $\xi_i : \partial_{\infty} \pi_1(S) \to \PP^1 \RR$ be the corresponding equivariant boundary maps for $1\leq i \leq d$, and let for all $n\in \NN$
$$
\left \{ \begin{array}{rl} 
\slack_{i,n} :=& \slack_{\xi_i}(\ell_n, \ell'_n) \vspace{1mm} \\
\boldsymbol{X}_n :=& ( \slack_{i,n} )_{1\leq i \leq d} \end{array} \right .
\quad \text{and} \quad 
\left \{ \begin{array}{rl} 
\slack'_{i,n} :=& \slack_{\xi_i}(\ell_n, -\ell'_n) \vspace{1mm} \\
\boldsymbol{X}'_n :=& ( \slack'_{i,n} )_{1\leq i \leq d}~. \end{array} \right .
$$
Lemma~\ref{lem:thinpants}, applied coordinate-wise, then yields resolutions $\gamma'_n$ and $\gamma''_n, \gamma'''_n$ such that
\begin{align*}
\boldsymbol{\lambda_\rho}(\gamma_n'') + \boldsymbol{\lambda_\rho}(\gamma_n''') & = \boldsymbol{\lambda_\rho}(\gamma_n) - 2\boldsymbol{X}_n + o(1) \\
\boldsymbol{\lambda_\rho}(\gamma_n') & = \boldsymbol{\lambda_\rho}(\gamma_n) - 2 \boldsymbol{X}'_n + o(1) ~.
\end{align*}

Let $c>0$ be given by Lemma~\ref{lem:Omega-non-pos}: after possibly passing to a subsequence, we have $\Omega(\boldsymbol{X}_n) \geq c \, \slack_{1,n} $ for all $n\in\NN$, or $\Omega(\boldsymbol{X}'_n) \geq c \, \slack'_{1,n} $ for all~$n$.

We treat the latter case first: we have
\begin{align} \notag
\frac{\Omega(\boldsymbol{\lambda_\rho}(\gamma_n') )}{\lambda_{\rho_1}(\gamma_n')} &= \frac{\Omega(\boldsymbol{\lambda_\rho}(\gamma_n) ) - 2\,\Omega(\boldsymbol{X}'_n) + o(1)}{\lambda_{\rho_1}(\gamma_n) - 2\,\slack'_{1,n} + o(1)}\\ \notag
&\leq \frac{\Omega(\boldsymbol{\lambda_\rho}(\gamma_n) ) - 2c\, \slack'_{1,n}  + o(1)}{\lambda_{\rho_1}(\gamma_n) - 2\,\slack'_{1,n}  + o(1)}\\ 
\label{eq:mna}
&= m_n - \frac{2 (c - m_n)\, \slack'_{1,n}  + o(1)}{\lambda_{\rho_1}(\gamma_n) - 2\,\slack'_{1,n}  + o(1)}
\end{align}
where $m_n = \Omega(\boldsymbol{\lambda_\rho}(\gamma_n))/\lambda_{\rho_1}(\gamma_n)$ as in \eqref{eq:record}.
Note that $m_n \to 0 $, and $\slack'_{1,n} \to \slack_{\xi_1}(\ell, -\ell') > 0$ by continuity of the cross ratio. 
Therefore \eqref{eq:mna} is eventually $<m_n$. 
Since $\lambda_{\rho_1}(\gamma_n') < \lambda_{\rho_1}(\gamma_n)$, this contradicts the record breaker property~\eqref{eq:record}.

In the remaining case, we similarly obtain that
$$\frac{\Omega(\boldsymbol{\lambda_\rho}(\gamma_n'') )+\Omega(\boldsymbol{\lambda_\rho}(\gamma_n''') )}{\lambda_{\rho_1}(\gamma_n'') + \lambda_{\rho_1}(\gamma_n''')}$$
is eventually $<m_n$, hence $\min \left \{ \frac{\Omega(\boldsymbol{\lambda_\rho}(\gamma_n'') )}{\lambda_{\rho_1}(\gamma_n'')}, \frac{\Omega(\boldsymbol{\lambda_\rho}(\gamma_n''') )}{\lambda_{\rho_1}(\gamma_n''')} \right \} < m_n$, again contradicting~\eqref{eq:record}.

Thus the geodesic current $\mu$ is in fact a measured geodesic lamination: $\mu\in \mathscr{ML}(S)$.
Therefore $\mathcal{H} = \mathrm{Ker}(\Omega)$ meets
the set $\bslambda_{\bsrho}(\mathscr{ML}(S))$, hence its convex hull in~$\mathfrak{a}^+$, which is the simple hull $\Lambda_{\bsrho}^s$ (Definition~\ref{def:simple-hull}).
This completes the proof of Theorem~\ref{thm:azimut-limit-cone-switched}, hence the ``only if'' direction of Theorem~\ref{thm:azimut-limit-cone}.

%%%%%%%%%%%%%%%%%%%%%%%%%
\subsection{Proof of Theorem~\ref{thm:azimut-limit-cone}}

For the other direction, let $\mathcal{H}$ be an azimuthal supporting hyperplane to the simple hull~$\Lambda_{\bsrho}^s$: we have $\mathcal{H} = \mathrm{Ker}(\Omega)$ and $\Omega(\Lambda_{\bsrho}^s)\geq 0$ where $\Omega = \sum_{i=1}^d c_i e_i^*$ for some $c_i\in\RR$ which are all $\geq 0$ except for one, which we can take to be $c_1<0$ up to renumbering.
For any $0 < t < |c_1|$, if we set $\Omega_t := \Omega + t e_1^*$, then $\Omega_t(x) > \Omega(x)$ for all $x \in \mathfrak{a}^+ \smallsetminus \{0\}$. 
For $t$ close enough to $|c_1|$, we have $\Omega_t(\Lambda_{\bsrho} \smallsetminus \{0\}) > 0$. 
It follows by continuity that there exists $0 \leq t < |c_1|$ such that $\mathrm{Ker}(\Omega_t)$ is a supporting hyperplane to $\Lambda_{\bsrho}$.
This supporting hyperplane is still azimuthal.
By Theorem~\ref{thm:azimut-limit-cone-switched}, there exists $\mu \in \mathscr{ML}(S)$ such that $\Omega_t(\bslambda_{\bsrho}(\mu)) = 0$. 
Therefore $t = 0$ and $\Omega(\Lambda_{\bsrho}) = \Omega_t(\Lambda_{\bsrho}) \geq 0$. 
This completes the proof of Theorem~\ref{thm:azimut-limit-cone}.

%%%%%%%%%%%%%%%%%%%%%%%%%%%%%%%%%%%%%%%%%%%%%%%%%%%
\section{Finite-sided limit cones} \label{sec:finite-sided}

In this section we use Theorem~\ref{thm:azimut-limit-cone} to construct examples of surfaces $S$ and multi-Fuchsian representations $\boldsymbol{\rho}: \pi_1(S) \to (\PSL_2 \RR)^3$ for which the projectivized limit cone $\PP \Lambda_{\boldsymbol \rho}$ is a polygon, and we prove Theorem~\ref{thm:N-gon}.

We use the following terminology.

\begin{definition} \label{def:azim-polygon}
A convex polyhedron $\Pi$ of $\PP(\RR^d_{\geq 0})$ is \emph{azimuthal} if it lifts to a convex cone $\widetilde{\Pi}$ of $\RR^d_{\geq 0}$ such that for any facet of~$\Pi$, the corresponding supporting hyperplane to~$\widetilde{\Pi}$ is azimuthal in the sense of Definition~\ref{def:azim}.
\end{definition}

Such a polyhedron always contains $[1:\dots:1]$ in its interior.

Before giving a general construction of multi-Fuchsian representations with polygonal projectivized limit cone, let us first demonstrate the basic idea in the case of a three-holed sphere.

\begin{example} \label{ex:pants}
Let $S$ be a three-holed sphere, with fundamental group $\pi_1(S) = \langle a,b,c \,|\, cba=\nolinebreak 1\rangle$ where $a,b,c$ correspond to the three boundary curves.
Famously, $S$ has only three simple closed geodesics, namely the three boundary curves $a,b,c$, and the support of any measured lamination on~$S$ is contained in their union.
In each hyperbolic structure on~$S$ with holonomy~$\rho_i$, the three lengths $\lambda_{\rho_i}(a), \lambda_{\rho_i}(b),\lambda_{\rho_i}(c)\geq 0$ may be chosen independently; these determine $\rho_i$ up to conjugacy. 
Hence given choices of three vectors $v_a, v_b, v_c \in (\RR_{>0})^3$, there is a unique (up to conjugacy) multi-Fuchsian representation $\bsrho: \pi_1(S) \to (\PSL_2 \RR)^3$ with no cusps such that $(\bslambda_{\bsrho}(a), \bslambda_{\bsrho}(b), \bslambda_{\bsrho}(c)) = (v_a, v_b, v_c)$.
Its simple hull (Definition~\ref{def:simple-hull}) is $\Lambda_{\bsrho}^s = \RR_{\geq 0}\,v_a + \RR_{\geq 0}\,v_b + \RR_{\geq 0}\,v_c$ --- a cone on a triangle.
It is easy to choose $v_a, v_b, v_c$ so that the triangle of $\PP(\RR^3_{\geq 0})$ with vertices $[v_a], [v_b], [v_c]$ is azimuthal: \eg $v_a = (t_a,1,1)$, $v_b = (1,t_b,1)$, and $v_c = (1,1,t_c)$ work for any $t_a,t_b,t_c\in (0,1)$.
In this case, Theorem~\ref{thm:azimut-limit-cone} implies that the projectivized limit cone $\PP \Lambda_{\boldsymbol \rho}$ is this projective triangle.
\end{example}

%%%%%%%%%%%%%%%%%%%%%%%%%
\subsection{Strategy of the proof of Theorem~\ref{thm:N-gon}} \label{subsec:strategy-N-gon}

Let $g\geq 2$. 
A right-angled convex hyperbolic $(2g+2)$-gon $P$, doubled across all its sides, determines a closed hyperbolic surface $S_P$ of genus~$g$, made of $4$ isometric copies of $P$.
Namely, the group generated by the reflections in the sides of $P$ admits $\pi_1(S_P)$ as a normal subgroup of index $4$, and the quotient is isomorphic to $(\ZZ/2\ZZ)^2$ which acts on $S_P$ by isometries. 
Concretely, one may first double $P$ across its even sides to obtain a sphere with $g+1$ holes, and then across the boundaries of the holes to obtain $S_P$.
We choose a numbering $1,\dots,2g+2$ of the edges of~$P$ in counterclockwise (\ie positive) cyclic order, which we call a \emph{cyclic edge labelling} of~$P$.

If $P'$ is another right-angled convex hyperbolic $(2g+2)$-gon, then any cyclic edge labelling of~$P'$ determines an isomorphism $\pi_1(S_P) \simeq \pi_1(S_{P'})$, hence a holonomy representation $\rho_{P'} : \pi_1(S_P) \to \PSL_2\RR$ for the hyperbolic surface $S_{P'}$.
We will focus on multi-Fuchsian representations of the form $\boldsymbol{\rho} = (\rho_{P_1}, \rho_{P_2}, \rho_{P_3}) : \pi_1(S_P) \to (\PSL_2 \RR)^3$ where $P_1, P_2, P_3$ are right-angled $(2g+2)$-gons with cyclic edge labellings.
Due to the common $(\ZZ/2\ZZ)^2$ symmetry, any direction in the limit cone $\Lambda_{\boldsymbol{\rho}}$ is approached by Jordan projections of geodesic currents on~$S_P$ that are symmetric under $(\ZZ/2\ZZ)^2$, \ie under all reflections.
Among these symmetric geodesic currents, consider those that do not self-intersect, \ie the symmetric geodesic laminations.
They are supported on unions of lifts to~$S_P$ of \emph{edges} and \emph{chords} of~$P$, where a chord is a segment perpendicular to sides of $P$ at both ends. 
An edge of $P$ lifts in $S_P$ to a twice-longer loop.
A chord of $P$ lifts to either two twice-longer loops (if its endpoints lie in two sides with the same parity), or a single four-times-longer loop (otherwise). 
There are $2g+2$ edges and $(g+1)(2g-3)$ chords in total.
Thus the projectivized simple hull $\PP \Lambda_{\bsrho}^s$ of~$\bsrho$ is a finite polygon in $\PP \mathfrak{a}^+ = \PP(\RR^3_{\geq 0})$.
For well-chosen right-angled hyperbolic $(2g+2)$-gons $P_1, P_2, P_3$ with cyclic edge labellings, it will be possible to show that this polygon is azimuthal; it will then be equal to the projectivized limit cone $\PP \Lambda_{\bsrho}$ by Theorem~\ref{thm:azimut-limit-cone}.

\begin{remark} \label{rem:number-sides}
Although the number of sides of the polygon $\PP \Lambda_{\bsrho}$ could in principle become as large as quadratic in $g$ with this technique, in our actual arguments the vertices of $\PP \Lambda_{\boldsymbol{\rho}}$ will always be given by sides and chords that are \emph{short} in one of the three representations. 
Up to $2g-1$ sides and chords can become short in a right-angled $(2g+2)$-gon, so even then it may be possible to push the bound in Theorem~\ref{thm:N-gon} up to $6g-3$; we have not attempted to optimize it.
On the other extreme, Theorem~\ref{thm:N-gon} and a covering argument also imply that there exist multi-Fuchsian representations of closed surface groups of arbitrarily high genus whose projectivized limit cone $\PP \Lambda_{\boldsymbol{\rho}}$ is a polygon with constant number of sides.
\end{remark}

%%%%%%%%%%%%%%%%%%%%%%%%%
\subsection{Genus 2: hexagons} \label{sec:6-gon}

For expository reasons, we first demonstrate the technique in the case $g=2$, proving a slight variant of Theorem~\ref{thm:N-gon} (namely, $6$ sides instead of~$\geq7$).
Then in Section~\ref{subsec:g-gon} we prove Theorem~\ref{thm:N-gon} in any genus $g\geq 2$.  

\begin{proposition} \label{prop:hexagon}
Let $S_2$ be a closed orientable surface of genus~$2$.
There exist multi-Fuchsian representations $\boldsymbol{\rho} : \pi_1(S_2) \to (\PSL_2\RR)^3$ whose projectivized limit cone $\PP \Lambda_{\boldsymbol{\rho}}$ is a hexagon.
\end{proposition}

\begin{proof}
We will build our hexagons, and all subsequent $(2g+2)$-gons, out of right-angled pentagons.
A convex right-angled hyperbolic pentagon has side lengths, enumerated in counterclockwise cyclic order,
$$\textstyle \big (\arcsinh \, \frac{\cosh y}{\sinh x}~,~ x~,~\arccosh \, \frac{1}{\tanh x \tanh y}~,~ y ~,~ \arcsinh \, \frac{\cosh x}{\sinh y} \big )$$
where $x,y>0$ (see \eg \cite[Th.\,7.18.1]{bea95}).
We shall take both $x$ and~$y$ to be very small, in which case this becomes
\begin{align*}
& \textstyle \big ( |\log \frac{x}{2}| ~,~ x~,~|\log \frac{xy}{2}|~,~ y ~,~ |\log \frac{y}{2}| \big )~ +o(1) \\
= ~& \big ( X ~,~ x~,~X+Y~,~ y ~,~ Y \big )~ +O(1)
\end{align*}
where $X:=|\log x| \gg 1$ and $Y:=|\log y| \gg 1$, and where the error term (here and below) applies only to the capital letters.

Let $P'=P'(x,y)$ be such a pentagon.
Let $P=P(x,y)$ be the right-angled hexagon obtained as the union of $P'$ and its symmetric image across the midpoint of the edge of length $y$.
By construction, $P$ is centrally symmetric, and we can number its edges $1,\dots,6$ in a counterclockwise cyclic order so that the corresponding set of edge lengths is
$$ \big (x~,~ X+2Y~,~ X~,~ x~,~ X+2Y~,~ X \big )~ + O(1). $$
The chords $1\mathrm{-}4$, $2\mathrm{-}5$, $3\mathrm{-}6$ then have lengths, respectively,
$$ \big (2X+2Y~,~ y~,~ 2Y)~ + O(1). $$
We now take $P_1$ to be $P$ with this cyclic edge labelling, and $P_2$ (\resp $P_3$) to be $P$ with the cyclic edge labelling where all edge indices have been shifted by~$1$ (\resp $2$) modulo 6.
Omitting the $O(1)$ error terms, we can now list all lengths:
$$\begin{array}{c|ccc|ccc}
&  \hspace{-1cm} \mathrlap{\text{edges (symmetric pairs):}} &&&& \text{chords:} & \\ 
& 1 \text{ or  } 4 & 2 \text{ or  } 5 & 3 \text{ or  } 6& 1\mathrm{-}4 & 2\mathrm{-}5 & 3\mathrm{-}6 \\ \hline
\overset{~}{P_1} & x & X+2Y & X & 2X+2Y & y & 2Y \\
P_2 &X & x & X+2Y & 2Y & 2X+2Y & y \\
P_3 & X+2Y & X & x & y & 2Y & 2X+2Y \\
\end{array}$$
The columns of this table are (unnormalized) barycentric coordinates in the projectivized Weyl chamber $\PP \mathfrak{a}^+ = \PP(\RR^3_{\geq 0})$.
They form the vertices of a hexagon $\Pi$ with order-3 rotational symmetry, whose vertices are all very close to sides of $\PP \mathfrak{a}^+$ (because $x,y \ll 1 \ll X,Y$).
For $y=x\ll1$, azimuthality \emph{fails}: the side of $\Pi$ connecting the third and fourth (projectivized) columns of the table runs close and near-parallel to the side $[*:*:0]$ of the triangle $\PP\mathfrak{a}^+$, cutting off two vertices instead of one.
Instead, we take $y=\delta x$ for some $\delta \in \RR_{>0} \smallsetminus [2/3,4]$: then $Y=X+O(1)$, and one easily checks that the hexagon $\Pi$ is azimuthal when $x\ll 1$.

As in Section~\ref{subsec:strategy-N-gon}, let $\boldsymbol{\rho} = (\rho_{P_1}, \rho_{P_2}, \rho_{P_3}) : \pi_1(S_P)\to (\PSL_2\RR)^3$ be the multi-Fuchsian representation of $\pi_1(S_P)$ associated to the triple $(P_1,P_2,P_3)$ of right-angled hyperbolic hexagons with cyclic edge labellings.
By construction, the projectivized simple hull $\PP \Lambda_{\bsrho}^s$ of~$\bsrho$ is the azimuthal hexagon~$\Pi$.
It is equal to the full projectivized limit cone $\PP \Lambda_{\bsrho}$ by Theorem~\ref{thm:azimut-limit-cone}.
\end{proof}

\smallskip

Multiplicatively adjusting, as in the proof just above, a short chord $x$ or $y$ to move the corresponding vertex of $\Pi$  towards or away from the nearest side of $\PP \mathfrak{a}^+$, without much affecting its position along that side (recorded by $X,Y$), will be useful again in arbitrary genus to generate convexity and azimuthality, as we now explain.

%%%%%%%%%%%%%%%%%%%%%%%%%
\subsection{Arbitrary genus: $(2g+2)$-gons} \label{subsec:g-gon}

We now prove Theorem~\ref{thm:N-gon} in arbitrary genus $g\geq 2$.
Let $P$ be a right-angled convex hyperbolic $(2g+2)$-gon with a cyclic edge labelling.
We can decompose $P$ into a chain of $2g-2$ right-angled pentagons $Q_1$, \dots, $Q_{2g-2}$, with $Q_k$ and $Q_{k+1}$ sharing exactly one edge $e_k$ for $1\leq k < 2g-2$.
The segments $e_k$ are chords of~$P$.
We will assume that for each $1<k<2g-2$, in the pentagon~$Q_k$, the side $e_k$ comes two sides clockwise after $e_{k-1}$ if $k$ is even, counterclockwise if $k$ is odd.
We extend this definition to $k=1$ and $k=2g-2$ to define edges $e_0$ of $Q_1$ and $e_{2g-2}$ of $Q_{2g-2}$.
We will use the lengths $x_0, \dots, x_{2g-2}$ of $e_0, \dots, e_{2g-2}$ as parameters for the deformation space of the cyclically labelled right-angled $(2g+2)$-gon $P=P(x_0, \dots, x_{2g-2})$.
When all the $x_k$ are small, any chord or edge $c$ of $P$ that is not an $e_k$ has length
$$ \psi_c (|\log x_0|,~\dots~, |\log x_{2g-2}|)+ O(1)$$ 
for some linear form $\psi_c \in (\RR^{2g-1})^*$ with coefficients in~$\NN$.
(The $O(1)$ error terms depend on the genus $g$.)
See Figure~\ref{fig:zigzag}.

\begin{figure}[h!] 
\captionsetup{width=0.9\linewidth}
\labellist
\small\hair 2pt
% 		G R I D    F O R    P L A C I N G    L A B E L S
%\grille{0} \grille{2} \grille{4} \grille{6} \grille{8} \grille{10} \grille{12} \grille{14} \grille{16} \grille{18}
%\grille{20} \grille{22} \grille{24} \grille{26} \grille{28} \grille{30} \grille{32} \grille{34} \grille{36} \grille{38} 
%\grille{40} \grille{42} \grille{44} \grille{46} \grille{48} \grille{50} \grille{52} \grille{54} \grille{56} \grille{58} 
%
\pinlabel {${}_{Q_1}$} [c] at			4.2	6.8
\pinlabel {${}_{Q_2}$} [c] at			10.5	2.3
\pinlabel {${}_{Q_{2g-2}}$} [c] at		45.4	2.3
\pinlabel {${}_{x_0}$} [c] at			0.5	4.1
\pinlabel {${}_{x_1}$} [c] at			7.95	4.2
\pinlabel {${}_{x_{2g-2}}$} [c] at			50.2	4.6
\pinlabel {$c$} [c] at					17.6	6.8
\pinlabel {$(\cdot \cdot \cdot)$} [c] at		31.5	4.5
\pinlabel {\rotatebox{45}{${}_{\sim |\log x_0|}$}} [c] at		1.4	7.7
\pinlabel {\rotatebox{315}{${}_{\sim |\log x_1|}$}} [c] at	6.9	7.6
\pinlabel {\rotatebox{315}{${}_{\sim |\log x_1|}$}} [c] at	8	1.5
%%%%%%%%%%%%%%%%%%%%%%%%%
\endlabellist
\includegraphics[width = \textwidth]{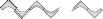}
\caption{A thin right-angled convex hyperbolic $(2g+2)$-gon $P(x_0, \dots, x_{2g-2})$, with a chord $c$ (dotted)}
\label{fig:zigzag}
\end{figure}
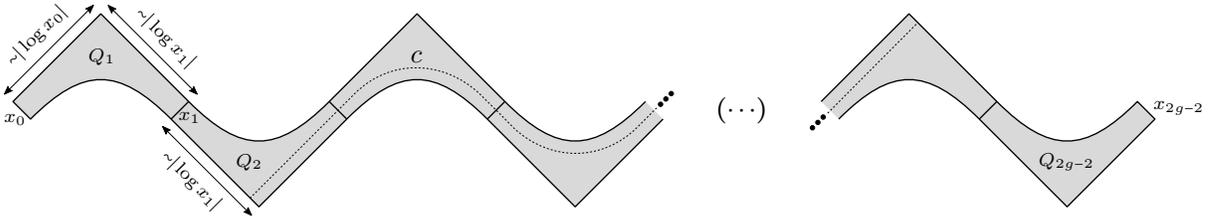

For $j\in\{1,2\}$, we denote by
$$P[j] = P(x_0, \dots, x_{2g-2})[j]$$
the cyclically labelled right-angled convex hyperbolic $(2g+2)$-gon obtained from $P$ by shifting all indices of the cyclic edge labelling by~$j$ units.
Fix positive reals $\alpha_0, \dots, \alpha_{2g-2}$ such that all the $\beta_c := \psi_c(\alpha_0, \dots, \alpha_{2g-2})$ are pairwise incommensurable (\eg take the $\alpha_k$ independent over $\QQ$).
Next, choose a small number $x>0$ and define three right-angled $(2g+2)$-gons:
$$ P_1^\circ:=P(x,\dots, x) \quad \quad
P_2^\circ:=P(x,\dots, x)[1] \quad \quad
P_3^\circ:=P(x^{\alpha_0}, \dots, x^{\alpha_{2g-2}})[2]. $$
The shifts in the cyclic edge labelling will cause any chord or edge of $P_i^\circ$ that is short (length $x$ or $x^{\alpha_k}$) to become long in the other two $P_\ell^\circ$.
Ultimately $P_1, P_2, P_3$ will be a small perturbation of $P_1^\circ, P_2^\circ, P_3^\circ$, chosen to ensure azimuthality. 

\smallskip

Here are the details; we refer to Figure~\ref{fig:manygon}.
Let $X:=|\log x| \gg 1$.
For each $0\leq k\leq 2g-2$, the lengths of $e_k$ in the metrics $P_1^\circ, P_2^\circ, P_3^\circ$ form a positive triple of the form
$$v_k = \begin{pmatrix} x \\ \nu_k X + O(1) \\ \beta_k X + O(1) \end{pmatrix} \quad \text{ as } x \to 0,$$
where $\nu_k$ is a fixed positive integer and $\beta_k = \psi_{c_k}(\alpha_0, \dots, \alpha_{2g-2})$ a fixed linear combination of $\alpha_0, \dots, \alpha_{2g-2}$ with coefficients in $\NN$.
The positive ratios $\beta_k/\nu_k$ are all distinct, because the $\beta_k$ are pairwise incommensurable.

Similarly, the $2g-1$ short edges and chords of $P_2^\circ$ have length triples
$$v'_k = \begin{pmatrix}  \nu'_k X + O(1) \\ x \\  \beta'_k X + O(1) \end{pmatrix}$$
where $\nu'_k$ is a fixed positive integer and $\beta'_k$ a fixed linear combination of $\alpha_0, \dots, \alpha_{2g-2}$ over~$\NN$.
Again, the positive ratios $\beta'_k/\nu'_k$ are all distinct.

Finally, the $2g-1$ short edges and chords of $P_3^\circ$ have length triples
$$v''_k = \begin{pmatrix}  \nu''_k X + O(1) \\ \nu'''_k X + O(1) \\  x^{\alpha_k} \end{pmatrix}$$
where $\nu''_k$ and $\nu'''_k$ are fixed positive integers. 
The ratios $\nu''_k/\nu'''_k$ may not be all distinct, but one of the bottom entries $x^{\alpha_{k_0}}$ is much smaller than all others as $x \to 0$.

\begin{figure}[h!] 
\captionsetup{width=0.9\linewidth}
\labellist
\small\hair 2pt
% 		G R I D    F O R    P L A C I N G    L A B E L S
%\grille{0} \grille{2} \grille{4} \grille{6} \grille{8} \grille{10} \grille{12} \grille{14} \grille{16} \grille{18}
%\grille{20} \grille{22} \grille{24} \grille{26} \grille{28} \grille{30} \grille{32} \grille{34} \grille{36} \grille{38} 
%\grille{40} \grille{42} \grille{44} \grille{46} \grille{48} \grille{50} \grille{52} \grille{54} \grille{56} \grille{58} 
%
\pinlabel {${}_{[v_{\sigma(0)}]}$} [c] at			3.2	2
\pinlabel {${}_{[v_{\sigma(1)}]}$} [c] at			3.45	3
\pinlabel {${}_{[v_{\sigma(2)}]}$} [c] at			4	4.3
\pinlabel {$\rotatebox{45}{\dots}$} [c] at			4.65	5.85
\pinlabel {${}_{[v''_{k_0}]}$} [c] at				5.9	0.7
\pinlabel {${}_{[1:1:1]}$} [c] at					6	3.8
\pinlabel {$\PP \mathfrak{a}^+$} [c] at			3.2	0.4
\pinlabel {$\PP \Lambda_{\boldsymbol{\rho}}$} [c] at	7.5	1.5
\pinlabel {$I_{\rho_{P_1}}$} [c] at				2.5	5.5
\pinlabel {$I_{\rho_{P_2}}$} [c] at				9.5	5.5
\pinlabel {$I_{\rho_{P_3}}$} [c] at				8	-0.35
%%%%%%%%%%%%%%%%%%%%%%%%%
\endlabellist
\includegraphics[width = 8cm]{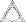}
\vspace{2mm}
\caption{The projectivized limit cone $\PP \Lambda_{\bsrho} = \PP \Lambda_{\bsrho}^s$ of a multi-Fuchsian representation $\boldsymbol{\rho} = (\rho_{P_1}, \rho_{P_2}, \rho_{P_3}) : \pi_1(S_P)\to (\PSL_2\RR)^3$ as in Section~\ref{subsec:g-gon}. It is a polygon with $\geq 4g-1$ sides. Here $g=3$.}
\label{fig:manygon}
\end{figure}

It follows from the above that the projectivizations $[v_0], \dots, [v_{2g-2}]$ converge to distinct interior points $[v_k^\infty]:=[0:\nu_k:\beta_k]$ of the first side $I_{\rho_{P_1}}$ of the triangle $\PP \mathfrak{a}^+$, as $x\rightarrow 0$. 
We can read out the limit points $[v_{\sigma(0)}^\infty], \dots, [v_{\sigma(2g-2)}^\infty]$ in linear order along $I_{\rho_{P_1}}$, for an appropriate permutation $\sigma$ of $\{0,\dots, 2g-2\}$.
Moreover, the $[v_k]$ are all roughly at the same distance $\frac{x}{X}=\frac{x}{|\log x|} \ll 1$ from the segment $I_{\rho_{P_1}}$, up to factors bounded in terms of $g$.
Therefore, up to adjusting the short edges and chords of the polygon $P_1^\circ$ by bounded multiplicative amounts (which does not affect the limits $[v_k^\infty]$), we may ensure that the polygonal line $([v_{\sigma(0)}], \dots, [v_{\sigma(2g-2)}])$ in $\PP \mathfrak{a}^+$ bounds a convex polygon which is azimuthal (Definition~\ref{def:azim-polygon}).
Let $P_1$ be this adjustment of $P_1^\circ$.

Similarly, adjusting $P_2^\circ$ to~$P_2$, we may ensure that the $[v'_k]$ define a polygonal line bounding an azimuthal convex polygon near the second side $I_{\rho_{P_2}}$ of the projectivized Weyl chamber $\PP \mathfrak{a}^+$.
Up to scaling the initial tuple $(\alpha_0, \dots, \alpha_{2g-2})$, which essentially scales the third coordinate of the limiting vectors $[v_k^\infty]=[0:\nu_k:\beta_k]$ and~$[{v'_k}^\infty]=[\nu'_k:0:\beta'_k]$, we may also assume that the center $[1:1:1]$ of the Weyl chamber $\PP\mathfrak{a}^+$ lies in the interior of the convex hull of $\{[v_k]\}_{i=0}^{2g-2} \cup \{[v'_k]\}_{i=0}^{2g-2}$.

These adjustments do not change the property that one of the $[v''_k]$ is much closer to the third side $I_{\rho_{P_3}}$ than all others, causing the convex hull of all the $[v_k], [v'_k], [v''_k]$ to be an azimuthal convex polygon $\Pi$ with at least $2(2g-1)+1=4g-1$ vertices.
The Jordan projection of any edge or chord that is short in \emph{none} of $P_1, P_2, P_3=P_3^\circ$ is at least some constant away from the sides of $\PP \mathfrak{a}^+$, and therefore cannot decrease the number of sides or destroy azimuthality of~$\Pi$.
Thus, $\Pi = \PP \Lambda_{\bsrho}^s \subset \PP \mathfrak{a}^+$ is an azimuthal convex polygon with at least $4g-1$ sides.

As in Section~\ref{subsec:strategy-N-gon}, let $\boldsymbol{\rho} = (\rho_{P_1}, \rho_{P_2}, \rho_{P_3}) : \pi_1(S_P)\to (\PSL_2\RR)^3$ be the multi-Fuchsian representation of $\pi_1(S_P)$ associated to the triple $(P_1,P_2,P_3)$ of right-angled hyperbolic $(2g+\nolinebreak 2)$-gons with cyclic edge labellings.
By construction, the projectivized simple hull $\PP \Lambda_{\bsrho}^s$ of~$\bsrho$ is the azimuthal convex polygon~$\Pi$.
This polygon is equal to the full projectivized limit cone $\PP \Lambda_{\bsrho}$ by Theorem~\ref{thm:azimut-limit-cone}.
This completes the proof of Theorem~\ref{thm:N-gon}.

\begin{remark}
The method extends to $(\PSL_2\RR)^d$ for $d>3$, with $d-1$ different shifts of cyclic edge labellings, to make polyhedra $\Pi:=\mathrm{Conv}\{[\boldsymbol{\lambda_\rho}(\gamma)]\}_{\gamma \in \,  \mathsf{Chords} \, \cup \, \mathsf{Edges}}$ with at least $(d-1)(2g-1)+1$ vertices when $d\leq g+1$. 
However, in this setup every chord and edge is short in at most one component $\rho_{P_i}$, so every $[\boldsymbol{\lambda_\rho}(\gamma)]$ is close to at most one facet of the simplex~$\PP\mathfrak{a}^+$. 
To conclude without further analysis that $\Pi$ must be azimuthal (and therefore equal to $\PP \Lambda_{\boldsymbol{\rho}}$) would require some $[\boldsymbol{\lambda_\rho}(\gamma)]$ near every \emph{edge} of $\PP\mathfrak{a}^+$. 
Modifying the $\rho_{P_i}$ in such a fashion may be possible, but we do not pursue that goal here. 
\end{remark}

%%%%%%%%%%%%%%%%%%%%%%%%%%%%%%%%%%%%%%%%%%%%%%%%%%%
\section{Strictly convex limit cones} \label{sec:infinite-sided}

In this section we take $S$ to be a one-holed torus, so that the fundamental group $\pi_1(S)$ is a nonabelian free group $\mathbb{F}_2$ of rank two.
Our goal is to construct multi-Fuchsian representations $\boldsymbol{\rho} : \pi_1(S) \rightarrow (\PSL_2\RR)^3$ such that every simple closed curve in~$S$ defines an extremal point in $\PP \Lambda_{\boldsymbol{\rho}}$, thus proving Theorem~\ref{thm:fishy}.

For this we start by considering a (Fuchsian) representation $\rho : \pi_1(S)\to\PSL_2 \RR$ which is the holonomy representation of an infinite-volume complete hyperbolic structure on~$S$.
The simple closed curves in~$S$ can be generated iteratively from one another by taking certain products in $\pi_1(S)$, in a process we explain below.
Our first result (Proposition~\ref{prop:honey}) will say that $\rho$-lengths behave \emph{very slightly sub-additively} as this process advances, with an accurate estimate of the defect.
Then in Proposition~\ref{prop:toycone}, we quantitatively show that $\rho$-length defines a norm with \emph{strictly} convex unit ball on the first homology group $\mathrm{H}_1(S, \RR) \simeq \RR^2$. 
This can be connected to estimates of \cite{zag82, gue-toulouse} and stable norm in negative curvature~\cite{bangert, baba}; see also Mirzakhani's curve counting results in~\cite{mir08}.
In Proposition~\ref{prop:homoclinic} we use a version of the same estimates for three representations at once in order to prove Theorem~\ref{thm:fishy}.

%%%%%%%%%%%%%%%%%%%%%%%%%
\subsection{Markoff maps}

For the punctured torus $S$, as shown by Bowditch~\cite{bow98}, a representation $\hat{\rho} : \pi_1(S)=\mathbb{F}_2 \to \SL_2\RR$ can be captured by a \emph{Markoff map}, which is the assignment of a real number $t_Q$ to each complementary component $Q \subset \mathbb{H}^2$ of the infinite trivalent Markoff tree, or Farey tree.
(This tree is dual to the Farey triangulation, obtained from the triangle $(01\infty) \subset \HH^2$ by iterated reflections in its sides: see Figure~\ref{fig:farey}.)
Namely,
\begin{equation} \label{eqn:t_Q}
t_Q:= \mathrm{Tr} (\rho(\gamma_Q))
\end{equation}
where $\gamma_Q \in \mathbb{F}_2$ is a word representing the simple closed curve of $S\simeq (\RR^2\smallsetminus \ZZ^2)/\ZZ^2$ whose \emph{slope} is the unique rational of $\PP^1 \QQ \subset \PP^1 \RR \simeq \partial_\infty \mathbb{H}^2$ lying at infinity of $Q$.
The trace identity $\mathrm{Tr}(x)\mathrm{Tr}(y)=\mathrm{Tr}(xy)+\mathrm{Tr}(xy^{-1})$ in $\SL_2\RR$ ensures that for any two adjacent regions $A,B$ sharing common neighbors $C$ and $D$, we have $t_A t_B = t_C+t_D$.
In this way, the Markoff map $t$ is entirely determined by its values $(t_A, t_B, t_C)$ on a triple of mutually adjacent regions.
When $\hat{\rho}$ is a lift (from $\PSL_2\RR$ to $\SL_2\RR$) of the holonomy representation of an infinite-volume complete hyperbolic structure on~$S$, we have $|t_Q|>2$ for all regions~$Q$, and we may in fact choose the lift so that $t_Q>2$ for all~$Q$, see~\cite{bow98}.
We shall assume this from now on.

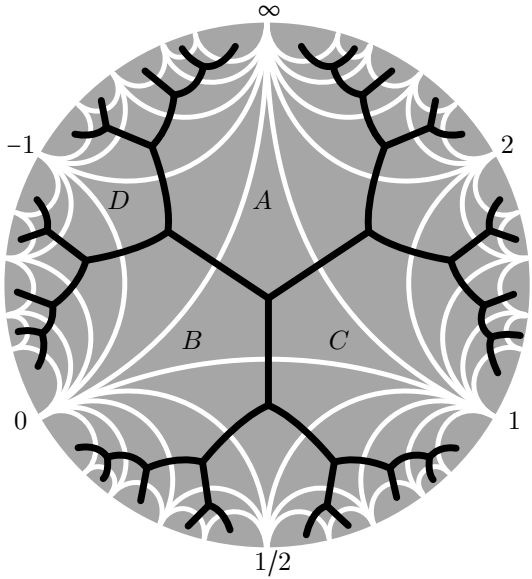
\begin{figure}[h!]
\captionsetup{width=0.9\linewidth}
\labellist
\small\hair 2pt
\pinlabel {$0$} [c] at     	0.5 		3.7
\pinlabel {$1$} [c] at        	14.6 		3.7
\pinlabel {$\infty$} [c] at   	7.6 		15.4
\pinlabel {$-1$} [c] at       	0.5 		11.6
\pinlabel {$1/2$} [c] at      	7.7 		-0.4
\pinlabel {$2$} [c] at        	14.4 		11.6
\pinlabel {$A$} [c] at        	7.4 		10
\pinlabel {$B$} [c] at        	5.4 		6.0
\pinlabel {$C$} [c] at        	9.6 		6.0
\pinlabel {$D$} [c] at        	3.3 		10
\endlabellist
\includegraphics[width = 7cm]{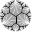}
\caption{The Markoff tree (black) superimposed on the Farey triangulation (white) in the hyperbolic plane (gray)}
\label{fig:farey}
\end{figure} 

For shorthand, let us use the names of the regions $A,B,C,D$ for the reals $t_A, t_B, t_C, t_D$ themselves.
Let $a,b,c,d>0$ be such that
\begin{equation*}
(2\cosh(a), 2\cosh(b), 2\cosh(c), 2\cosh(d))=(A,B,C,D).
\end{equation*}
Let $T \in \RR$ be the trace of the commutator of (any) two elements of $\mathbb{F}_2$ that form a free basis, such as $(\gamma_A, \gamma_B)$ for example. 
The Markoff identity (see \cite[p.~338, formula~(7)]{fk97}) stipulates
\begin{equation} \label{eq:markoff}
T=A^2+B^2+C^2-ABC-2
\end{equation}
(note that $T$ depends only on the composition $\rho: \pi_1(S) \overset{\hat{\rho}}{\longrightarrow} \mathrm{SL}_2 \RR \rightarrow \mathrm{PSL}_2 \RR$, and is also invariant under all outer automorphisms of $\mathbb{F}_2$).
The Markoff identity \eqref{eq:markoff} can be rewritten
\begin{equation} \label{eq:markoff2}
 2-T = \frac{(\mathrm{e}^{a+b}-\mathrm{e}^c) (\mathrm{e}^{b+c}-\mathrm{e}^a) (\mathrm{e}^{c+a}-\mathrm{e}^b) (\mathrm{e}^{a+b+c}-1)}{\mathrm{e}^{2(a+b+c)}}.
 \end{equation}
The group of outer automorphisms of $\mathbb{F}_2$ identifies with $\SL_2\ZZ$, acting naturally on the Markoff tree.

%%%%%%%%%%%%%%%%%%%%%%%%%
\subsection{Hyperbolic triangles and slacks}

Consider a hyperbolic triangle $\tau \subset \mathbb{H}^2$ with side lengths $a,b,c>0$.
The group generated by the $\pi$-rotations around the vertices of $\tau$ has an index-two subgroup $F_\tau$ generated by the hyperbolic translations $\alpha,\beta,\gamma$ along the sides of $\tau$, of respective lengths $2a$, $2b$, $2c$.
We have $\gamma\beta\alpha = 1$, and generically this is the only relation, hence $F_\tau$ is a nonabelian free group of rank two.
Using notation \eqref{eqn:t_Q}, we can identify $([\alpha^{\pm 1}], [\beta^{\pm 1}], [\gamma^{\pm 1}])$ with $([\gamma_A^{\pm 1}], [\gamma_B^{\pm 1}], [\gamma_C^{\pm 1}])$ for three pairwise adjacent complementary regions $A,B,C$ of the Markoff tree: such triples $\{A,B,C\}$ stand in natural bijection with Markoff tree vertices, and also with Farey triangles.
The \emph{slack} between the generators $\alpha$ and~$\beta$, by which we mean the sum of their lengths minus the length of their product, is twice the defect of the triangle inequality:
$$\mathsf{slack} = 2(a + b - c).$$
By \eqref{eqn:slack-rope}, this slack approaches the asymptotic slack \eqref{eqn:slack-Gromov-bound2} of the oriented geodesic axes of $\alpha$ and~$\beta$ when the sides $a$ and $b$ are very large and form small angles with the third side.

It is a hyperbolic trigonometry exercise to check that the translation length $\ell =\linebreak 2\,\arccosh(-T/2)$ of the commutator $\alpha \beta \alpha^{-1} \beta^{-1}$ satisfies 
\begin{equation} \label{eq:trigo}
\cosh (\ell/4) = \sinh (c) \, \sinh (h_c) ,
\end{equation} 
where $h_c$ is the height of $\tau$ perpendicular to the side~$c$.
(When the commutator $\alpha\beta\alpha^{-1}\beta^{-1}$ is parabolic or elliptic of angle $\theta \in [0,2\pi)$, we can set $\ell =\mathbf{i} \theta$ and $\cosh(\ell/4) = \cos(\theta/4) \in (0,1]$.
Topologically, $F_\tau \backslash \mathbb{H}^2$ is a one-holed torus if $\ell \geq 0$.)

A change of generator pair in $F_\tau$ will correspond to other hyperbolic triangles, all yielding the same value of~\eqref{eq:trigo}, hence these triangles become exponentially thin ($h_c \rightarrow 0$) as their largest side $c$ increases.
Thus, $\mathsf{slack}=2(a+b-c)$ also decreases exponentially as $c\to +\infty$.

Specifically, if $\tau'$ is obtained from the triangle $\tau$ by replacing one of its vertices by its symmetric image with respect to another vertex (and keeping the third vertex fixed), then $F_{\tau'}=F_\tau$: we call $\tau'$ a \emph{mutation} of~$\tau$. 
A mutation corresponds to a step in the Markoff tree, \eg from the vertex $\overline{A}\cap \overline{B}\cap \overline{C}$ to $\overline{A}\cap \overline{B}\cap \overline{D}$.
For example in the left panel of Figure~\ref{fig:lanes} below, the dark grey triangle $\tau_0$ mutates (in one of three possible ways) to the yellow one~$\tau_1$.

Proposition~\ref{prop:honey} below will formalize the idea that mutations cause triangles to flatten and slacks to decrease, very fast.

%%%%%%%%%%%%%%%%%%%%%%%%%
\subsubsection{Thin triangles} \label{sec:thintriangles}

Choose a small $\varepsilon >0$.
There exists a small $\delta >0$ with the following property: for any triangle $\tau$ in $\mathbb{H}^2$ with all three angles $\leq \delta$, if $a, b, c$ are the side lengths of~$\tau$ (necessarily all large), then $\mathrm{e}^a\leq \varepsilon \mathrm{e}^{b+c}$ and $\mathrm{e}^b \leq \varepsilon \mathrm{e}^{c+a}$ and $\mathrm{e}^c \leq \varepsilon \mathrm{e}^{a+b}$ and $1 \leq \varepsilon \mathrm{e}^{a+b+c}$. 
If $\tau$ is such a triangle, then~\eqref{eq:markoff2} yields
\begin{equation}
(1-\varepsilon)^4\,\mathrm{e}^{a+b+c} \leq 2-T \leq \mathrm{e}^{a+b+c} . \label{eq:313}
\end{equation}
Next, consider a sequence $\tau=\tau_0, \tau_1, \tau_2, \dots$ obtained by a succession of mutations (with no backtracking).
By an immediate induction, for every $n\geq 1$, the largest angle of the triangle $\tau_n$ is $\geq \pi-\delta$ and is an increasing function of $n$, limiting to $\pi$ as $n \rightarrow \infty$. 
The side lengths also grow under mutation.
Therefore, if we call $(a_n, b_n, c_n)$ the side lengths of~$\tau_n$, with $c_n$ the largest of the three, then up to taking $\delta$ small enough we have for all $n\geq 1$
\begin{align}
(1-\varepsilon)\,\mathrm{e}^{a_n+b_n} & \leq \mathrm{e}^{c_n} \leq \mathrm{e}^{a_n+b_n} , \label{eq:314} \\ 
\mathrm{e}^{b_n} \leq \varepsilon \mathrm{e}^{a_n+c_n} \quad \text{and} \quad & \mathrm{e}^{a_n} \leq \varepsilon \mathrm{e}^{b_n+c_n} \quad \text{and} \quad 1 \leq \varepsilon \mathrm{e}^{a_n+b_n+c_n} .
\label{eq:315}
\end{align}
By \eqref{eq:markoff2}--\eqref{eq:315} we have $ (1-\varepsilon)^3 (\mathrm{e}^{a_n+b_n} - \mathrm{e}^{c_n}) \mathrm{e}^{c_n} \leq 2-T \leq (\mathrm{e}^{a_n+b_n} - \mathrm{e}^{c_n}) \mathrm{e}^{c_n}$, hence
$$ (2-T) \mathrm{e}^{-2c_n} \,\leq\, \mathrm{e}^{a_n+b_n-c_n} -1 \,\leq\, (1-\varepsilon)^{-3} \, (2-T) \mathrm{e}^{-2c_n}.$$
Further bounding $2-T$ and $\mathrm{e}^{c_n}$ by~\eqref{eq:313} and~\eqref{eq:314}, we find
\begin{equation} \label{eq:316}
(1-\varepsilon)^4 \, \mathrm{e}^{a_0+b_0+c_0-2(a_n+b_n)} \,\leq\, \mathrm{e}^{a_n+b_n-c_n} -1 \,\leq\, (1-\varepsilon)^{-5} \, \mathrm{e}^{a_0+b_0+c_0-2(a_n+b_n)} .
\end{equation}
Since $a_1$ and $b_1$ belong to $\{a_0, b_0, c_0\}$, the bounds~\eqref{eq:316} are at most of the order of $\varepsilon$ for $n=1$, and much smaller for $n\geq 2$, and the slack $2(a_n+b_n-c_n)$ behaves similarly, as summarized by the following proposition (see Figure~\ref{fig:lanes}).

\begin{proposition} \label{prop:honey}
For every $\varepsilon>0$, there exists $\delta>0$ such that for any hyperbolic triangle $\tau_0$ with inner angles $\leq \delta$, and any non-backtracking sequence of mutations $(\tau_0, \tau_1, \tau_2 \dots)$, if $\tau_n$ has side lengths $a_n \leq b_n \leq c_n$, then 
$$ c_n = a_n + b_n - K \, \mathrm{e}^{-2(a_n +b_n)} \, (1+\varepsilon_n)$$  
for every $n\geq 1$, where $K:=\mathrm{e}^{a_0 + b_0 + c_0}$ and $|\varepsilon_n| \leq \varepsilon$.
Moreover, the ``slack'' term\linebreak $K \, \mathrm{e}^{-2(a_n +b_n)} \, (1+\varepsilon_n)$ is bounded by $\varepsilon$ for all $n\geq 1$ (not just large $n$). \qed
\end{proposition}

\begin{figure}[h!]
\captionsetup{width=0.9\linewidth}
\labellist
\small\hair 2pt
% 		G R I D    F O R    P L A C I N G    L A B E L S
%\grille{0} \grille{2} \grille{4} \grille{6} \grille{8} \grille{10} \grille{12} \grille{14} \grille{16} \grille{18}
%\grille{20} \grille{22} \grille{24} \grille{26} \grille{28} \grille{30} \grille{32} \grille{34} \grille{36} \grille{38} 
%\grille{40} \grille{42} \grille{44} \grille{46} %\grille{48} \grille{50} \grille{52} \grille{54} \grille{56} \grille{58} 
%
\pinlabel {$\tau_0$} [c] at		3.7	4.6
\pinlabel {${}_{\tau_1}$} [c] at	7.8	3.9
%%%%%%%%%%%%%%%%%%%%%%%%%
\endlabellist
\includegraphics[width = \textwidth]{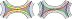}
\caption{\small Illustration of Proposition~\ref{prop:honey}.
 {\bf Left}: a triangle $\tau_0 \subset \mathbb{H}^2$ with side lengths $a_0, b_0, c_0$ and  small angles.
We show the result of a single mutation $\tau_1$ (yellow), and a fundamental domain of the convex core of $F_{\tau_0}=F_{\tau_1} < \PSL_2\RR$ (light gray, opposite short sides identified).
Shortest simple loops in $F_{\tau_0} \backslash \mathbb{H}^2$ have lengths $2a_0, 2b_0, 2c_0$ (red, green, blue).
The length of the peripheral loop is close to $2(a_0+b_0+c_0)$, as per~\eqref{eq:313}.
\newline {\bf Right}: in dotted lines, a simple closed curve obtained after a few mutations, whose length is close to some even linear combination of $a_0, b_0, c_0$.}
\label{fig:lanes}
\end{figure}

%%%%%%%%%%%%%%%%%%%%%%%%%
\subsection{One representation}

Let $\mathbb{F}_2$ be a free group of rank $2$, seen as the fundamental group of a one-holed torus $S$.
For every rational $p/q$ in the rational projective line $\PP^1\QQ = \QQ \sqcup \{ 1/0 \}$, let $\gamma_{p/q} \in \mathbb{F}_2$ be an element representing a simple closed curve in $S$ of slope $p/q$ (unique up to conjugation and inversion). 
Given two rationals $p/q$ and $p'/q'$ in reduced form, the two curves $[\gamma_{p/q}], [\gamma_{p'/q'}]$ intersect exactly once if and only if $(p/q, p'/q')$ is a Farey edge, or equivalently,  $|pq'-qp'|=1$: then, the two adjacent Farey triangles have their tips at $(p+p')/(q+q')$ and $(p-p')/(q-q')$, and the corresponding simple closed curves $[\gamma_{(p+p')/(q+q')}]$, $[\gamma_{(p-p')/(q-q')}]$ are the two resolutions of the crossing pair $\{[\gamma_{p/q}], [\gamma_{p'/q'}]\}$.

Fix $\kappa>1$, for instance $\kappa=2$.
Choose a small $\varepsilon>0$, and $\delta=\delta(\varepsilon)>0$ as in Proposition~\ref{prop:honey}.
Let $\tau$ be a hyperbolic triangle with inner angles $\leq \delta$, and side lengths $a_{}, b_{}, c_{}$ (large).
For convenience, we make the extra assumption that
\begin{equation*}
\frac{\max \, \{a_{}, b_{}, c_{} \}}{ \min \, \{a_{}, b_{}, c_{} \}} \leq \kappa.
\end{equation*}
We will write $X \sim_\kappa Y$ to say that the ratio of two quantities $X,Y$ is bounded by two positive constants depending only on $\kappa$, not on $\varepsilon$: for example, $a_{}/b_{} \sim_\kappa 1$ but $a_{} \not\sim_\kappa 1$.
Let $\rho_\tau : \pi_1(S)=\mathbb{F}_2 \rightarrow \mathrm{PSL}_2\RR$ be the holonomy representation of an infinite-volume complete hyperbolic structure on~$S$ such that 
$$ \big ( \lambda_{\rho_\tau}(\gamma_{0/1})~,~ \lambda_{\rho_\tau}(\gamma_{1/1}) ~,~ \lambda_{\rho_\tau} (\gamma_{1/0}) \big) = (2a_{}, 2b_{}, 2c_{}).$$
In general, for coprime $0\leq p \leq q$ the length  
$\lambda_{\rho_\tau}(\gamma_{p/q})$ is roughly $2(p b_{} + (q-p) a_{})$ (up to a factor close to $1$ when $\varepsilon$ is small), hence
\begin{equation} \label{eq:coarsecone} \textstyle
\lambda_{\rho_\tau}(\gamma_{p/q}) ~ \sim_\kappa  ~ 2a_{} q.
\end{equation}

We construct a map $L_\tau: \RR^2 \rightarrow \RR_{\geq 0}$ as follows:
\begin{itemize}
\item $L_\tau(q,p) = \frac{1}{2a_{}} \, \lambda_{\rho_\tau} (\gamma_{p/q})$ for all coprime $(q,p)\in \ZZ^2$;
\item homogeneity: for every $t \in \RR$ we impose $L_\tau(t q, t p) = |t|\, L_\tau (q,p)$;
\item $L_\tau$ is continuous on $\RR^2$.
\end{itemize}
The normalizing factor $\frac{1}{2a_{}}$ is there to make the $L_\tau$ roughly $(\sim_\kappa 1)$-Lipschitz along rays.
Up to this factor, $L_\tau$ is simply the length functional associated to $\rho_\tau$, restricted to measured laminations (where a coprime pair $(q,p)\in \ZZ^2$ represents the Lebesgue measure on the closed geodesic of slope~$\frac{p}{q}$). 
Thus, continuity of $L_\tau$ follows from continuity of the length function over the space of currents.
Further, the fact that resolving crossings decreases geodesic length can be used to show that $L_\tau$ is convex.
We will prove the following slightly stronger statement.

\begin{proposition} \label{prop:toycone}
The function $L_\tau:\RR^2 \rightarrow \RR_{\geq 0}$ is convex, and for every coprime pair $(q,p)\in \ZZ^2$ the ray $\{L_\tau(tq,tp)\}_{t\geq 0}$ is an extremal ray of the region above the graph of $L_\tau$, with nonunique supporting plane.
\end{proposition}

\begin{proof}
We first restrict to rays with slopes in $[0,1]$.
Let $R:=\{(x,y)\in \RR^2~|~ 0\leq y \leq x\}$ be the union of all these rays.
Define increasing collections $\sigma_k \in [0,1]^{2^k+1}$ by: $\sigma_0 = ( 0/1, 1/1 )$ and, whenever $p/q, \, p'/q'$ are consecutive terms of $\sigma_k$, then $\sigma_{k+1}$ contains these same two terms, plus $(p+p')/(q+q')$ in between.
Every rational number of $[0,1]$ eventually shows up in $\sigma_k$ for $k$ large enough.
By construction, any two consecutive terms of $\sigma_k$ are Farey neighbors.
Define a function $L_\tau^k :R \rightarrow \RR_{\geq 0}$ as follows:
\begin{itemize}
  \item in restriction to any ray whose slope lies in $\sigma_k$, we take $L_\tau^k=L_\tau$;
  \item in the wedge between two consecutive such rays, extend $L_\tau^k$ linearly.
\end{itemize}
The intersection of the graph of $L_\tau^k$ with the plane at height $1$ in $\RR^3$ is a polygonal line $\mathscr{L}_{\tau}^k$ connecting the points 
$$\xi_{p/q} := \frac{2a_{}}{\lambda_{\rho_\tau}(\gamma_{p/q})} (q,p) $$
 for $p/q$ ranging over $\sigma_k$ (here, $0\leq p \leq q$ and $p,q$ are coprime).
Note that $\xi_{0/1}=(1,0)$ and $\xi_{1/1}=(u,u)$ for some $u\sim_\kappa 1$.
For other $p/q\in [0,1]$, the $\RR^2$-norm of $\xi_{p/q}$ is also $\sim_\kappa 1$, due to~\eqref{eq:coarsecone}.

The polygonal line $\mathscr{L}_{\tau}^k$ is obtained from $\mathscr{L}_{\tau}^{k-1}$ by ``adding a small triangle'' on each segment of $\mathscr{L}_{\tau}^{k-1}$.
Specifically, if $\sigma_{k-1} = (\dots, p'/q', p''/q'', \dots)$ and $\sigma_k = (\dots, p'/q', p/q, p''/q'', \dots)$ with $p=p'+p''$ and $q=q'+q''$, then we refer to the triangle $\xi_{p'/q'} \xi_{p/q} \xi_{p''/q''}$ as~$\Delta_{p/q}$. 

We claim that each polygonal line $\mathscr{L}_{\tau}^k \subset \RR^2$ is convex, and very close to the segment $\mathscr{L}_{\tau}^0$ from $\xi_{0/1}$ to $\xi_{1/1}$.
To see this, we must estimate the exterior angle $\theta_{p/q}$ of the triangle $\Delta_{p/q}$ at its tip $\xi_{p/q}$, and show that the $\theta_{p/q}$ are small and decay rapidly enough (in an auxiliary Euclidean metric on $\RR^2$, see Figure~\ref{fig:fishy}).

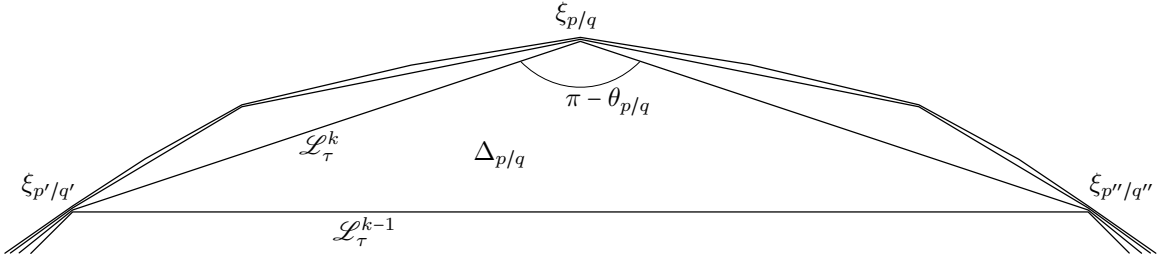
\begin{figure}[h!] 
\captionsetup{width=0.9\linewidth}
\labellist
\small\hair 2pt
% 		G R I D    F O R    P L A C I N G    L A B E L S
%\grille{0} \grille{2} \grille{4} \grille{6} \grille{8} \grille{10} \grille{12} \grille{14} \grille{16} \grille{18}
%\grille{20} \grille{22} \grille{24} \grille{26} \grille{28} \grille{30} \grille{32} \grille{34} \grille{36} \grille{38} 
%\grille{40} \grille{42} \grille{44} \grille{46} %\grille{48} \grille{50} \grille{52} \grille{54} \grille{56} \grille{58} 
%
\pinlabel {$\mathscr{L}_{\tau}^{k-1}$} [c] at			8	0.5
\pinlabel {$\mathscr{L}_{\tau}^k$} [c] at			7	2.4
\pinlabel {$\xi_{p/q}$} [c] at		12.7	5.3
\pinlabel {$\xi_{p'/q'}$} [c] at		1	1.5
\pinlabel {$\xi_{p''/q''}$} [c] at		24.8	1.5
\pinlabel {$\Delta_{p/q}$} [c] at		11	2.2
\pinlabel {$\pi-\theta_{p/q}$} [c] at	13.4	3.4
%%%%%%%%%%%%%%%%%%%%%%%%%
\endlabellist
\includegraphics[width = \textwidth]{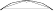}
\caption{Fast-converging sequence of polygonal lines $\mathscr{L}_{\tau}^k$}
\label{fig:fishy}
\end{figure}

We can use Proposition~\ref{prop:honey} to estimate $\theta_{p/q}$.
Namely, the angle $\measuredangle_{p/q}$  in $\RR^3$ between the ray through $\big ( q,p, \frac{1}{2a_{}}\, \lambda_{\rho_\tau}(\gamma_{p/q}) \big)$, belonging to the graph of $L_\tau^k$, and the graph of $L_\tau^{k-1}$ (restricted to the wedge $\smash{\{0 \leq \frac{p'}{q'}x \leq y \leq \frac{p''}{q''} x\}}$), satisfies 
\begin{equation}
\label{eq:smallangle1}
\measuredangle_{p/q} ~\sim_\kappa ~ \left . \frac{K \, \mathrm{e}^{- \lambda_{\rho_\tau} (\gamma_{p'/q'}) - \lambda_{\rho_\tau} (\gamma_{p''/q''})}}{ 2 a_{}} \right / \big \Vert (q,p, {\textstyle \frac{1}{2a_{}}} \lambda_{\rho_\tau}(\gamma_{p/q})) \big \Vert
 \end{equation}
with $K$ defined as in Proposition~\ref{prop:honey}. 
While 
$$K = \mathrm{e}^{a_{} + b_{} + c_{}} = \mathrm{e}^{\left (\lambda_{\rho_\tau}(\gamma_{0/1}) + \lambda_{\rho_\tau}(\gamma_{1/1}) + \lambda_{\rho_\tau} (\gamma_{1/0}) \right )/2}$$ 
\emph{does} depend on $\varepsilon$ (and is in fact quite large), it is small enough that already for the triple $(p'/q', p/q, p''/q'')=(0/1, 1/2, 1/1)$ the above bound is $\leq \varepsilon$, by Proposition~\ref{prop:honey} (recall also that $a_{}$ is large).
The norm in the denominator of~\eqref{eq:smallangle1} is $\sim_\kappa q$.
The angle $\measuredangle_{p/q}$ is also $\sim_\kappa h_{p/q}$, where $h_{p/q}$ is the \emph{height} of the triangle $\Delta_{p/q}$.
The distances of the tip of $\Delta_{p/q}$ to the other two vertices are $\sim_\kappa |p/q - p'/q'|$ and $\sim_\kappa |p/q-p''/q''|$, hence
$$ \theta_{p/q} ~ \sim_\kappa ~ \frac{h_{p/q}}{|\frac{p}{q}-\frac{p'}{q'}|} +  \frac{h_{p/q}}{|\frac{p}{q}-\frac{p''}{q''}|} = h_{p/q} \, (qq'+qq'') = h_{p/q} \, q^2.$$
Using $h_{p/q} \sim_\kappa \measuredangle_{p/q}$ and putting everything together, we find that 
\begin{equation}
\label{eq:smallangle2}
\theta_{p/q} ~\sim_\kappa~ \frac{K}{2a_{}} \,  \mathrm{e}^{-\lambda_{\rho_\tau} (\gamma_{p/q})} q
\end{equation}
for all rational slopes $p/q \in (0,1)$.

As we continue through the family of polygonal lines $\{ \mathscr{L}_{\tau}^{k+s} \}_{s\geq 1}$, the exterior angle of $\mathscr{L}_{\tau}^{k+s}$ at the tip $\xi_{p/q}$ of the triangle $\Delta_{p/q}$ decreases by successive amounts $\leq \theta_{p_s/q_s}$, where $(q_s,p_s)=(q',p')+s(q,p)$ or $(q'',p'')+s(q,p)$, due to insertion of triangles~$\Delta_{p_s/q_s}$. 
By~\eqref{eq:smallangle2}, these sequences $(\theta_{p_s/q_s})$ decay quickly with $s$; in fact their sums are dominated by twice the $s=1$ term, which is already much smaller than $\theta_{p/q}$.
Thus, $\xi_{p/q}$ remains a salient corner in every polygonal line $\mathscr{L}_{\tau}^{k+s}$, and in the limit $\mathscr{L}_{\tau}^{\infty}$ which is a convex arc.

\smallskip

Similar arguments apply for slopes in $(-\infty,0)$ and in $(1,+\infty)$.
This shows that the functions $L_\tau^k$ limit (from above) to a convex function~$L_\tau$. 
In fact, we have also shown that the graph of the restriction of $L_\tau$ to any rational ray is an \emph{angular} extremal ray of the convex cone $\{z\geq L_\tau(x,y)\} \subset \RR^3$, in the sense that it has nonunique supporting plane.
\end{proof}

%%%%%%%%%%%%%%%%%%%%%%%%%
\subsection{Three representations} \label{subsec:multi-Fuchsian-infinite-sided}

We move on to multi-Fuchsian representations of the form $\boldsymbol{\rho}=(\rho_1, \rho_2, \rho_3) : \pi_1(S)=\mathbb{F}_2 \rightarrow (\mathrm{PSL}_2 \RR )^3$; see Figure~\ref{fig:pointy}.
We continue with $\kappa, \varepsilon, \delta$ as in the previous section, but choose $\boldsymbol{\rho}$ such that
\begin{align*} \big ( 
\lambda_{\rho_1}(\gamma_{0/1}),~ 
\lambda_{\rho_1}(\gamma_{1/1}),~ 
\lambda_{\rho_1}(\gamma_{1/0}) \big) & = 2(a_{}, b_{}, b_{}) \\ \big ( 
 \lambda_{\rho_2}(\gamma_{0/1}),~ 
 \lambda_{\rho_2}(\gamma_{1/1}),~ 
 \lambda_{\rho_2}(\gamma_{1/0}) \big) & = 2(b_{}, a_{}, b_{}) \\\big ( 
 \lambda_{\rho_3}(\gamma_{0/1}),~ 
 \lambda_{\rho_3}(\gamma_{1/1}) ,~ 
 \lambda_{\rho_3}(\gamma_{1/0}) \big) & = 2(b_{}, b_{}, a_{})
\end{align*}
where $0 < a_{} < b_{}$ are large numbers.
Note that a large isosceles triangle $\tau$ in $\mathbb{H}^2$, with side lengths $a_{}, b_{}, b_{}$, \emph{always} has all three inner angles~$\leq \delta$. 
If $b_{}/a_{}$ becomes large (which requires a large $\kappa$), then the points $[\boldsymbol{\lambda_\rho} (\gamma_{p/q})]$ for $p/q\in\{0/1,1/1,1/0\}$ fall close to the midpoints of the walls of the projectivized Weyl chamber $\PP \mathfrak{a}^+$ of $(\mathrm{PSL}_2 \RR )^3$. 
However, we may continue to use $\kappa=2$.

\begin{figure}[h!] %    S I G N S
\captionsetup{width=0.9\linewidth}
\labellist
\small\hair 2pt
% 		G R I D    F O R    P L A C I N G    L A B E L S
%\grille{0} \grille{2} \grille{4} \grille{6} \grille{8} \grille{10} \grille{12} \grille{14} \grille{16} \grille{18}
%\grille{20} \grille{22} \grille{24} \grille{26} \grille{28} \grille{30} \grille{32} \grille{34} \grille{36} \grille{38} 
%\grille{40} \grille{42} \grille{44} \grille{46} %\grille{48} \grille{50} \grille{52} \grille{54} \grille{56} \grille{58} 
%
\pinlabel {$0$} [c] at				5.7	5.45
\pinlabel {$1$} [c] at				6.4	6.5
\pinlabel {$\infty$} [c] at			4.9	6.5
\pinlabel {${}_{1/2}$} [c] at			6.2	4
\pinlabel {${}_2$} [c] at			7.15	7.9
\pinlabel {${}_{-1}$} [c] at			3.3	6.8
\pinlabel {$\mathbb{H}^2$} [c] at	6	1.3
\pinlabel {$\rho_1$} [c] at		19.6		2.3
\pinlabel {$\rho_2$} [c] at		22.75	7.85
\pinlabel {$\rho_3$} [c] at		16.2		8
\pinlabel {$\tau$} [c] at		23.2		1.35
\pinlabel {$\tau'$} [c] at		21.9		11.6
\pinlabel {$\tau''$} [c] at		13.7		5.6
\pinlabel {$0$} [c] at			19.93	4
\pinlabel {$1$} [c] at			21.7		7
\pinlabel {$\infty$} [c] at		17.4		6.85
\pinlabel {${}_{1/2}$} [c] at		21.2		5.1
\pinlabel {${}_2$} [c] at		19.55	7.8
\pinlabel {${}_{-1}$} [c] at		18		5.2
\pinlabel {$\color{red}{}_{a_{}}$} [c] at			15.75	1.4
\pinlabel {$\color{blue}{}_{b_{}}$} [c] at			16.55	2.2
\pinlabel {$\color{vert}{}_{b_{}}$} [c] at			16.55	0.5
\pinlabel {$\PP \mathfrak{a}^+$} [c] at			16.6	4.3
\pinlabel {$\PP \Lambda_{\boldsymbol{\rho}}$} [c] at	19.1	6.4
%%%%%%%%%%%%%%%%%%%%%%%%%
\endlabellist
\includegraphics[width = \textwidth]{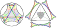}
\caption{\small {\bf Left}: a portion of the universal cover of a hyperbolic one-holed torus. 
Two right-angled hexagons (yellow and gray) form a fundamental domain of the convex core, and we draw lifts of simple loops of slopes $0,1,\infty$ (thick) and $-1,\frac{1}{2},2$ (dotted).
\newline {\bf Right}: the limit cone $\PP\Lambda_{\boldsymbol{\rho}}$ in the projectivized Weyl chamber $\PP \mathfrak{a}^+$, for a multi-Fuchsian representation $\bsrho = (\rho_1, \rho_2, \rho_3)$ as in Section~\ref{subsec:multi-Fuchsian-infinite-sided}, with pictures of corresponding fundamental domains. 
We respect the color codes. 
Hyperbolic lengths are (scaled) barycentric coordinates in the triangle $\PP \mathfrak{a}^+$: \eg the red dot marked~$0$ is closer to the bottom edge of $\PP \mathfrak{a}^+$ (corresponding to $\rho_1$) because the red curve, of slope $0$, is shorter in $\rho_1$ than in $\rho_2, \rho_3$.}
\label{fig:pointy}
\end{figure}

By the results of the previous section, to a first approximation each component $\lambda_{\rho_i}$ extends roughly linearly to slopes in $[0,1]$ (and similarly in $[-\infty, 0]$ and $[1,+\infty]$): 
$$ \boldsymbol{\lambda_\rho} (\gamma_{p/q}) = \boldsymbol{\lambda_\rho} (\gamma_{p'/q'}) + \boldsymbol{\lambda_\rho} (\gamma_{p''/q''}) +o(1)$$
for a pair of Farey neighbors $(p'/q', p''/q'')$, where $(q,p)=(q'+q'',p'+p'')$.
This time however, the defect $o(1)$ has (small, negative) components in each of the three dimensions. 
However, Proposition~\ref{prop:honey} says that one of these components (roughly $2K \mathrm{e}^{-\lambda_{\rho_i}(\gamma_{p/q})}$) vastly outweighs the other two, namely the one associated to the smallest component of $\boldsymbol{\lambda_\rho}(\gamma_{p/q})$ (\ie the nearest side to $[\boldsymbol{\lambda_\rho}(\gamma_{p/q})]$ in $\partial \PP \mathfrak{a}^+$). 
Therefore, inside each sector of the Weyl chamber, we can analyze $\boldsymbol{\lambda_\rho}$ in the same way as in the proof of Proposition~\ref{prop:toycone}, approaching from the inside the convex hull of its projectivized image by a sequence of convex polygonal lines $(\mathscr{L}^k)_{k\in \NN}$. 
The subdivision $[-\infty,-1] \cup [-1,0] \cup [0,1/2] \cup [1/2,1] \cup [1,2] \cup [2,+\infty]$ of $\PP^1\QQ$ is well suited to this analysis: by symmetry we may restrict to just one of these six intervals, and its associated Jordan projections stay in a single sector; see Figure~\ref{fig:pointy}.
Projectivizing, we obtain the following.

\begin{proposition} \label{prop:homoclinic}
The map $j_{\boldsymbol{\rho}}:\PP^1 \QQ \rightarrow \PP \mathfrak{a}^+$, taking $p/q$ to $[\boldsymbol{\lambda_\rho}(\gamma_{p/q})]$, extends continuously to $\PP^1 \RR$. 
Its image is the boundary of a strictly convex region $\PP \Lambda_{\bsrho}^{ss}$, of which every $[\boldsymbol{\lambda_\rho}(\gamma_{p/q})]$ is an angular point. 
When the lengths $a_{}$ and $b_{}$ are large, then $\PP \Lambda_{\bsrho}^{ss}$ is close to the triangle with vertices $[a_{}: b_{}: b_{}]$, $[b_{}: a_{}: b_{}]$, $[b_{}: b_{}: a_{}]$, with nearly-as-pointy corners. \qed
\end{proposition}

The convex set $\PP \Lambda_{\bsrho}^{ss}$ in the above statement coincides with the projectivized simple hull $\PP\Lambda_{\bsrho}^s$ (Definition~\ref{def:simple-hull}), because the only simple curve in the one-holed torus $S$ other than the $[\gamma_{p/q}]$ is the peripheral loop, whose Jordan projection lies at the center of the Weyl chamber.
Since every point of $\partial \PP \Lambda_{\bsrho}^{ss}$ is the (projectivized) Jordan projection of a measured lamination, and has a supporting direction that is azimuthal, Proposition~\ref{prop:homoclinic} and Theorem~\ref{thm:azimut-limit-cone} together imply 
$\PP \Lambda_{\bsrho}^{ss} =\PP \Lambda_{\bsrho}^s =\PP \Lambda_{\boldsymbol{\rho}}$, proving Theorem~\ref{thm:fishy}.

%%%%%%%%%%%%%%%%%%%%%%%%%%%%%%%%%%%%%%%%%%%%%%%%%%%
\section{Discontinuous behavior of the limit cone} \label{sec:lim-cone-discont}

Given a finite-type surface $S$ of negative Euler characteristic and $d\geq 2$, it is natural to ask how the limit cone $\Lambda_{\bsrho}$ varies as a function of $\bsrho \in \Hom(\pi_1(S),(\PSL_2\RR)^d)$.
The following observation is straightforward.

\begin{remark} \label{rem:lim-cone-lower-semicont}
The limit cone $\Lambda_{\bsrho}$ always varies lower semicontinuously: if a sequence $(\bsrho^{(n)})_{n\in\NN}$ converges to a representation $\bsrho$ in $\Hom(\pi_1(S),(\PSL_2\RR)^d)$, then any point $x\in\Lambda_{\bsrho}$ is a limit of points $x_n\in\Lambda_{\bsrho^{(n)}}$.
Indeed, by definition of~$\Lambda_{\bsrho}$ the projective class $[x]$ can be approached by a sequence $([\bslambda_{\bsrho}(\gamma_k)])_{k\in \NN}$ where $\gamma_k\in\pi_1(S)$ for all $k\in\NN$, and given $k\in\NN$ we have $\smash{\bslambda_{\bsrho^{(n)}}(\gamma_k) \underset{n\to +\infty}{\longrightarrow} \bslambda_{\bsrho}(\gamma_k)}$ by continuity of~$\bslambda$; we conclude by a diagonal extraction argument.
\end{remark}

In the multi-Fuchsian case with no cusps, upper semicontinuity also holds.

\begin{proposition}
Let $\bsrho = (\rho_1,\dots,\rho_d) : \pi_1(S)\to(\PSL_2\RR)^d)$ be a multi-Fuchsian representation with no cusps, or more generally a representation such that each factor~$\rho_i$ is injective with convex cocompact image.
Then for any sequence $(\bsrho^{(n)})_{n\in\NN}$ of representations converging to~$\bsrho$ in $\Hom(\pi_1(S),(\PSL_2\RR)^d)$, the limit cones $\Lambda_{\bsrho^{(n)}}$ converge to $\Lambda_{\bsrho}$: the Hausdorff distance between the projectivizations $\PP \Lambda_{\bsrho^{(n)}}$ and $\PP \Lambda_{\bsrho}$ goes to zero.
\end{proposition}

\begin{proof}
If $\rho : \pi_1(S)\to\PSL_2\RR$ is an injective representation with convex cocompact image, then for any $\varepsilon>0$ there is a neighborhood $\mathcal{V}_{\rho,\varepsilon}$ of $\rho$ in $\Hom(\pi_1(S),\PSL_2\RR)$ such that 
$$|\lambda_{\rho'}(\gamma) - \lambda_\rho(\gamma)| \leq \varepsilon \, \lambda_\rho(\gamma)$$
for all $\rho'\in\mathcal{V}_{\rho,\varepsilon}$ and all $\gamma\in\pi_1(S)$.
Indeed, it suffices to take $\mathcal{V}_{\rho,\varepsilon}$ small enough that any two of its elements are the holonomies of two hyperbolic metrics on the convex core of $S$ that are $(1+\varepsilon)$-bi-Lipschitz to one another. 

For any $\varepsilon>0$, let $\mathcal{U}_{\varepsilon}:= \mathcal{V}_{\rho_1,\varepsilon} \times \dots \times \mathcal{V}_{\rho_d,\varepsilon} \subset \Hom(\pi_1(S),(\PSL_2\RR)^d)$. 
Then for any $\bsrho' \in \mathcal{U}_{\varepsilon}$,
$$\Vert \bslambda_{\bsrho'}(\gamma) - \bslambda_{\bsrho}(\gamma) \Vert \leq \varepsilon\sqrt{d} \, \Vert\bslambda_{\bsrho}(\gamma)\Vert$$
holds at all $\gamma\in\pi_1(S)$, hence $\PP \Lambda_{\bsrho'}$ is $\varepsilon$-close to $\PP \Lambda_{\bsrho}$ for the Hausdorff distance.
\end{proof}

Interestingly, upper semicontinuity of the limit cone $\Lambda_{\bsrho}$ can fail in the presence of cusps.
We now briefly discuss this phenomenon.

%%%%%%%%%%%%%%%%%%%%%%%%%
\subsection{Case $d=2$} \label{subsec:lim-cone-discont-d=2}

Here is an example where upper semicontinuity fails at a multi-Fuchsian representation $\bsrho$ with cusps in both factors.

\begin{lemma} \label{lem:lim-cone-discont-d=2}
Let $S$ be a three-holed sphere with fundamental group $\pi_1(S) = \langle a,b,c \,|\, cba=\nolinebreak 1\rangle$ where $a,b,c$ correspond to the three boundary curves.
For $i=1$ (\resp $i=2$), let $\rho_i : \pi_1(S)\to\PSL_2\RR$ be the representation given by
$$\rho_i(a) = \begin{pmatrix} 1 & 6\\ 0 & 1\end{pmatrix} \quad\text{\bigg(\resp} \begin{pmatrix} 1 & 4\\ 0 & 1\end{pmatrix} \text{\bigg)} \quad\mathrm{and}\quad \rho_i(b) = \begin{pmatrix} 1 & 0\\ -6 & 1\end{pmatrix}  \quad\text{\bigg(\resp} \begin{pmatrix} 1 & 0\\ -4 & 1\end{pmatrix} \text{\bigg)} ;$$
it is the holonomy representation of a complete hyperbolic structure on~$S$ with two cusps.
Then $\bsrho := (\rho_1,\rho_2) \in \Hom(\pi_1(S),(\PSL_2 \RR)^2)$ has Zariski-dense image in $(\PSL_2 \RR)^2$ and
\begin{enumerate}
  \item\label{item:discont-1} the limit cone $\Lambda_{\bsrho}$ is the $\RR_{\geq 0}$-span of $(t_1,t_2) := (4\,\arccosh(3),4\,\arccosh(2))$ and $(1,1)$;
  \item\label{item:discont-2} there is a sequence $(\bsrho^{(n)})_{n\in\NN}$ of multi-Fuchsian representations, converging to~$\bsrho$, such that the limit cones $\Lambda_{\bsrho^{(n)}}$ converge to the $\RR_{\geq 0}$-span of $(t_1,t_2)$ and $(0,1)$, which strictly contains $\Lambda_{\bsrho}$.
\end{enumerate}
\end{lemma}

\begin{proof}
Zariski-density follows from Lemma~\ref{lem:Goursat}.

\eqref{item:discont-1} We first observe that the simple hull $\Lambda^s_{\bsrho}$ (Definition~\ref{def:simple-hull}) is the ray $\RR_{\geq 0} (t_1, t_2)$, since among the three simple curves $a,b,c$ on~$S$, only $c$ has nonzero Jordan projection $\bslambda_{\bsrho}(c) = (t_1, t_2)$.
The full limit cone $\Lambda_{\bsrho}$ also contains the ray $\RR_{\geq 0} (1,1) = \lim_{k\to +\infty}\RR_{\geq 0}\bslambda_{\rho}(a^kb)$, hence the $\RR_{\geq 0}$-span of $(t_1, t_2)$ and $(1,1)$ by Lemma~\ref{lem:conv-lim-cone}. 
To show that this is the full limit cone, we use \cite[Th.\,1.3]{gk17}, which implies that Theorem~\ref{thm:azimut-limit-cone} holds in the case $d = 2$ for general multi-Fuchsian representations, possibly with cusps.
In particular, $\RR_{\geq 0}(t_1, t_2)$ lies in the boundary of the simple hull $\Lambda_{\bsrho}^s$ and has an azimuthal supporting hyperplane, hence it lies in the boundary of $\Lambda_{\bsrho}$. 
On the other side, if $\R_{\geq 0} (1,1)$ were contained in the interior of $\Lambda_{\bsrho}$, then both sides of the cone would have azimuthal supporting hyperplanes, hence we would have $\Lambda_{\bsrho} = \Lambda_{\bsrho}^{s}$, which is not the case.

\eqref{item:discont-2} For $n\geq 1$, let $\bsrho^{(n)} := (\rho_1,\rho_2^{(n)}) \in \Hom(\pi_1(S),(\PSL_2\RR)^2)$ where
$$\rho_2^{(n)} (a) = \begin{pmatrix} \cosh \frac{1}{n} & 4n \sinh \frac{1}{n} \\ \phantom{-}\frac{1}{4n} \sinh \frac{1}{n} & \cosh \frac{1}{n}\end{pmatrix} \quad\mathrm{and}\quad \rho_2^{(n)} (b) = \begin{pmatrix} \cosh \frac{1}{n} & \frac{-1}{4n} \sinh \frac{1}{n} \\ -4n \sinh \frac{1}{n} & \cosh \frac{1}{n}\end{pmatrix} .$$
Then $(\bsrho^{(n)})_{n\geq 1}$ converges to~$\bsrho$.
We have $\lambda_{\rho_2^{(n)}}(a) = \lambda_{\rho_2^{(n)}}(b) = 2/n \to 0$ and $\lambda_{\rho_2^{(n)}}(c) =: t_2^{(n)} \to t_2$.
The simple hull $\Lambda_{\bsrho^{(n)}}$ is the $\RR_{\geq 0}$-span of $(t_1, t_2^{(n)})$ and $(0,1)$, which, for sufficiently large~$n$, has azimuthal supporting hyperplanes on both sides. Hence, using~\cite[Th.\,1.3]{gk17} as above, we have $\Lambda_{\bsrho^{(n)}} = \Lambda_{\bsrho^{(n)}}^s$, which converges to the $\RR_{\geq 0}$-span of $(t_1, t_2)$ and $(0,1)$.
\end{proof}

In \cite[\S\,10.6]{gk17} we previously considered a variant of this example where $\rho_2^{(n)}$ converges instead to the constant representation $1_{\PSL_2\RR}$ (hence $\Lambda_{\bsrho} = \RR_{\geq 0} (1,0)$), while $\Lambda_{\bsrho^{(n)}}$ can be made to converge to the $\RR_{\geq 0}$-span of $(1,0)$ and $(1,\beta)$ for any $0\leq\beta\leq 1$. 
For $\beta<1$, this is achieved by taking small order-$n$ rotations
\begin{equation} \label{eq:deteriorate} 
\rho_2^{(n)}(a):= \begin{pmatrix} \cos \frac{\pi}{n} & n^\beta \sin \frac{\pi}{n} \\ 
\frac{-1}{n^\beta} \sin \frac{\pi}{n} & \cos \frac{\pi}{n}\end{pmatrix}
\quad \text{ and } \quad
\rho_2^{(n)}(b):= \begin{pmatrix} \cos \frac{\pi}{n} & \frac{1}{n^\beta} \sin \frac{\pi}{n} \\ 
-n^\beta \sin \frac{\pi}{n} & \cos \frac{\pi}{n}\end{pmatrix}~:
\end{equation}
the extremal slope $\beta+o(1)$ is then approached by elements $\bsrho^{(n)}(a^{\lfloor n/2 \rfloor} b^{\lfloor n/2 \rfloor})$. 
The case $\beta=1$ follows by extraction, or by taking $\smash{\rho_2^{(n)}(a), \rho_2^{(n)}(b)}$ equal to small unipotent elements.

As a further variant in the same vein, for any $0\leq \beta<1$, the nondiscrete representations
$$\rho_2^{(n)}(a):= \begin{pmatrix} \cosh \frac{1}{n} & n^\beta \sinh \frac{1}{n} \\ 
\frac{1}{n^\beta} \sinh \frac{1}{n} & \cosh \frac{1}{n}\end{pmatrix}
\quad \text{ and } \quad
\rho_2^{(n)}(b):= \begin{pmatrix} \cosh \frac{1}{n} & \frac{-1}{n^\beta} \sinh \frac{1}{n} \\ 
-n^\beta \sinh \frac{1}{n} & \cosh \frac{1}{n}\end{pmatrix} ,$$
yield again $\rho_2^{(n)} \to 1_{\PSL_2\RR}$ (hence $\Lambda_{\bsrho} = \RR_{\geq 0} (1,0)$), while $\Lambda_{\bsrho^{(n)}} = \RR_{\geq 0}^2 = \aaa^+$ for all~$n$.

%%%%%%%%%%%%%%%%%%%%%%%%%
\subsection{Case $d=3$} \label{subsec:lim-cone-discont-d=3}

We now give an example of discontinuous behavior where the limiting representation is multi-Fuchsian with some convex cocompact factors.

\begin{lemma} \label{lem:lim-cone-discont-d=3}
Let $S$ be a three-holed sphere with fundamental group $\pi_1(S) = \langle a,b,c \,|\, cba=\nolinebreak 1\rangle$ where $a,b,c$ correspond to the three boundary curves.
Let $\bsrho: \pi_1(S) \to (\PSL_2 \RR)^3$ be a multi-Fuchsian representation such that $(\bslambda_{\bsrho}(a), \bslambda_{\bsrho}(b), \bslambda_{\bsrho}(c)) = (v_a, v_b, v_c)$ where $v_a := (t_a,1,1)$ and $v_b := (1,t_b,1)$ and $v_c := (1,1,0)$ for some $t_a,t_b \in (0,1)$ (see Example~\ref{ex:pants}).
Then $\bsrho$ has Zariski-dense image in $(\PSL_2\RR)^3$ and
\begin{enumerate}
  \item\label{item:discont-1} the limit cone $\Lambda_{\bsrho}$ is the $\RR_{\geq 0}$-span of $v_a$, $v_b$, and~$v_c$ --- a cone on a triangle;
  \item\label{item:discont-2} there is a sequence $(\bsrho^{(n)})_{n\in\NN}$ of representations in $\Hom(\pi_1(S),(\PSL_2\RR)^3)$, converging to~$\bsrho$, such that the limit cones $\Lambda_{\bsrho^{(n)}}$ have a limit which contains the $\RR_{\geq 0}$-span of $v_a$, $v_b$, $(t_a,1,0)$, and $(1,t_b,0)$, hence which strictly contains $\Lambda_{\bsrho}$.
\end{enumerate}
\end{lemma}

\begin{proof}
Zariski-density follows from Lemma~\ref{lem:Goursat}.

\eqref{item:discont-1} The simple hull $\Lambda_{\bsrho}^s$ is $\RR_{\geq 0}\text{-span}(v_a,v_b,v_c) := \RR_{\geq 0}\,v_a + \RR_{\geq 0}\,v_b + \RR_{\geq 0}\,v_c$; it is contained in $\Lambda_{\bsrho}$ by Lemma~\ref{lem:conv-lim-cone}.
Let us check that $\Lambda_{\bsrho} \subset \RR_{\geq 0}\text{-span}(v_a,v_b,v_c)$.
We can write $\bsrho$ as a limit of multi-Fuchsian representations $\bssigma^{(n)}$ with no cusps such that $(v_a^{(n)}, v_b^{(n)}, v_c^{(n)})_{n\in\NN} := (\bslambda_{\bssigma^{(n)}}(a), \bslambda_{\bssigma^{(n)}}(b), \bslambda_{\bssigma^{(n)}}(c))_{n\in\NN}$ converges to $(v_a, v_b, v_c)$.
The projective triangle\linebreak $\PP(\RR_{\geq 0}\text{-span}(v_a,v_b,v_c))$ 
is azimuthal (Definition~\ref{def:azim-polygon}), hence so is $\PP(\RR_{\geq 0}\text{-span}(v_a^{(n)}, v_b^{(n)}, v_c^{(n)}))$ for all large enough~$n$.
By Theorem~\ref{thm:azimut-limit-cone} (see Example~\ref{ex:pants}), we have $\Lambda_{\bssigma^{(n)}} =\linebreak \RR_{\geq 0}\text{-span}(v_a^{(n)}, v_b^{(n)}, v_c^{(n)})$ for all large~$n$, hence $\Lambda_{\bsrho} \subset \RR_{\geq 0}\text{-span}(v_a, v_b, v_c)$ by Remark~\ref{rem:lim-cone-lower-semicont}.

\eqref{item:discont-2} Write $\bsrho = (\rho_1,\rho_2,\rho_3) : \pi_1(S)\to (\PSL_2\RR)^3$.
Let $(g_n)_{n\in\NN}$ be a sequence of elliptic elements of $\PSL_2\RR$ such that $g_n$ has order~$n$ and $g_n\to\rho_3(c)$.
(This exists: up to conjugation $\rho_3(c)$ is triangular unipotent and we can use matrices such as~\eqref{eq:deteriorate} with $\beta=1$.)
Let $\bsrho^{(n)} = (\rho_1,\rho_2,\rho_3^{(n)})$ where $\rho_3^{(n)}(a) = \rho_3(a)$ and $\rho_3^{(n)}(c) = g_n$.
Then $\bsrho^{(n)} \to \bsrho$ as $n\to +\infty$.
We have $\bslambda_{\bsrho^{(n)}}(a) = (t_a,1,1)=v_a$ and $\bslambda_{\bsrho^{(n)}}(b) =(1,t_b,1+o(1)) \to v_b$.
We claim that  $(t_a,1,0) \in \Lambda_{\bsrho^{(n)}}$ for all $n\geq 1$.
Indeed, since $\rho_3^{(n)}(c^n)=\mathrm{Id}$, one can check that
$$\frac{1}{2k} \, \bslambda_{\bsrho^{(n)}}\big([a^k,c^n]\big) \underset{k\to +\infty}{\longrightarrow} (\lambda_{\rho_1}(a),\lambda_{\rho_2}(a),0) = (t_a,1,0).$$
Similarly, by considering the commutators $[b^k,c^n]$ we see that $(1,t_b,0) \in \Lambda_{\bsrho^{(n)}}$.
We conclude using Lemma~\ref{lem:conv-lim-cone}.
\end{proof}

%%%%%%%%%%%%%%%%%%%%%%%%%
\subsection{Anosov representations and $\theta$-convexity} \label{subsec:DeyOh}

Let $G$ be a real semisimple linear Lie group with Lie algebra~$\g$. 
Let $\aaa$ be a Cartan subspace of~$\g$, let $\aaa^+$ be a closed Weyl chamber in~$\aaa$, and let $\Delta \subset \aaa^*$ be the corresponding set of simple restricted roots of $\aaa$ in~$\g$.
For each $\alpha\in \Delta$, let $\sigma_\alpha:\aaa \to \aaa$ be the reflection in the wall $\mathrm{Ker}(\alpha)$. 
Choose a subset $\varnothing \subsetneq \theta\subsetneq \Delta$, and let $W_\theta$ be the finite reflection group generated by $\{\sigma_\alpha\}_{\alpha\in\Delta\smallsetminus \theta}$.

Dey--Oh \cite{do25}, using results of Kapovich--Leeb--Porti \cite{klp25}, proved that for any word hyperbolic group~$\Gamma$ and any $\theta$-Anosov representation $\bsrho : \Gamma\to G$ (in the sense of Labourie), if the limit cone $\Lambda_{\bsrho} \subset \aaa^+$ is \emph{$\theta$-convex} (\ie $\bigcup_{w\in W_\theta}\, w\cdot\Lambda_{\bsrho}$ is connected and convex), then $\bsrho'\mapsto \Lambda_{\bsrho'}$ is continuous at~$\bsrho$.
In \cite[\S\,7]{do25}, they also observed that upper semicontinuity at~$\bsrho$ may fail if $\Lambda_{\bsrho}$ is not $\theta$-convex, giving an example $\bsrho$ valued in a proper algebraic subgroup of~$G$ (for which $\Lambda_{\bsrho}$ itself is not convex).

Lemma~\ref{lem:lim-cone-discont-d=3} gives the first example showing that upper semicontinuity at~$\bsrho$ may fail  even when $\bsrho$ has Zariski-dense image, when $\Lambda_{\bsrho}$ is convex but not $\theta$-convex.
Here $G=(\PSL_2\RR)^3$ and $\aaa^+ = \RR_{\geq 0}^3$: the reflections $\sigma_\alpha$ are in the three axis planes, with $\theta=\{e_1^*,e_2^*\}\subset \{e_1^*,e_2^*,e_3^*\}=\Delta$, and $\theta$-Anosov means that $\rho_1$ and~$\rho_2$ are convex cocompact.

%%%%%%%%%%%%%%%%%%%%%%%%%%%%%%%%%%%%%%%%%%%%%%%%%%%
%%%%%%%%%%%%%%%%%%%%%%%%%%%%%%%%%%%%%%%%%%%%%%%%%%%

 \end{document}